\documentclass{amsart}

\usepackage{graphicx,amsmath,amsthm, amssymb,amsxtra,tikz,combelow,MnSymbol,pgfplots,hyperref, enumitem, appendix}

\pgfplotsset{compat = newest}
\pdfmapfile{+sansmathaccent.map}

\tikzset{anchorbase/.style={baseline={([yshift=-0.5ex]current bounding box.center)}}}
\tikzset{
    partial ellipse/.style args={#1:#2:#3}{
        insert path={+ (#1:#3) arc (#1:#2:#3)}
    }
}
\usetikzlibrary{calc}
\usetikzlibrary{decorations.markings}
\usetikzlibrary{decorations.pathreplacing}
\usetikzlibrary{arrows,shapes,positioning}
\tikzstyle directed=[postaction={decorate,decoration={markings,
    mark=at position #1 with {\arrow{>}}}}]
\tikzstyle rdirected=[postaction={decorate,decoration={markings,
    mark=at position #1 with {\arrow{<}}}}]
    
\usetikzlibrary{cd,external}

\definecolor{UCDblue}{cmyk}{1,0.56,0,0.34}
\definecolor{UCDgold}{cmyk}{0,0.19,1,0.15}

\DeclareMathOperator{\JM}{JM}

\newcommand{\Z}{\mathbb{Z}}

\newcommand{\C}{\mathbb{C}}

\newcommand{\PP}{\mathbb{P}}

\newcommand{\x}{{\sf x}}

\newcommand{\cJ}{\mathcal{J}}

\newcommand{\cO}{\mathcal{O}}
\newcommand{\cK}{\mathcal{K}}

\newcommand{\diag}{\mathrm{diag}}
\newcommand{\Proj}{\mathrm{Proj}}

\newcommand{\Sp}{\mathrm{Sp}}

\newcommand{\sq}{\square}

\newcommand{\MVcomment}[1]{} 

\usepackage[utf8]{inputenc}

\newtheorem{thm}{Theorem}[section]
\newtheorem{cor}[thm]{Corollary}
\newtheorem{prop}[thm]{Proposition}
\newtheorem{lemma}[thm]{Lemma}
\newtheorem{conj}[thm]{Conjecture}

\theoremstyle{definition}
\newtheorem{fact}[thm]{Fact}
\newtheorem{defn}[thm]{Definition}

\newtheorem{example}[thm]{Example}

\theoremstyle{remark}
\newtheorem{rem}[thm]{Remark}

\title{Affine Springer Fibers and Generalized Haiman Ideals}
\author{Joshua P. Turner\\ with an appendix by Eugene Gorsky and Joshua P. Turner}
\address{Department of Mathematics
\\
University of California, Davis\\
Davis, CA 95616, USA and Department of Mathematics
\\
University of California, Davis\\
Davis, CA 95616, USA}
\email{jpturner@ucdavis.edu, egorskiy@math.ucdavis.edu}

\newcommand{\Gr}{\mathrm{Gr}}
\newcommand{\Gln}{\mathrm{GL}_n}
\newcommand{\Sln}{\mathrm{SL}_n}
\newcommand{\sln}{\mathfrak{sl}_n}
\newcommand{\gln}{\mathfrak{gl}_n}
\newcommand{\HHH}{\mathrm{HHH}}
\newcommand{\HY}{\mathrm{HY}}
\newcommand{\FT}{\mathrm{FT}}

\newcommand{\CP}{\mathcal{P}}
\newcommand{\ee}{\mathbf{e}}
\newcommand{\area}{\mathrm{area}}
\newcommand{\dinv}{\mathrm{dinv}}

\begin{document}

\begin{abstract}
    We compute the Borel-Moore homology of unramified affine Springer fibers for $\Gln$ under the assumption that they are equivariantly formal and relate them to certain ideals discussed by Haiman. For $n=3$, we give an explicit description of these ideals, compute their Hilbert series, generators and relations, and compare them to generalized $(q,t)$-Catalan numbers. We also compare the homology to the Khovanov-Rozansky homology of the associated link, and prove a version of a conjecture of Oblomkov, Rasmussen, and Shende in this case.
\end{abstract}

\maketitle

\tableofcontents

\section{Introduction}

In this paper, we study a broad class of affine Springer fibers and their equivariant Borel-Moore homology, using geometric and algebraic techniques. This homology is related to the Langlands program \cite{Langlands}, Hilbert schemes of singular curves \cite{MY,MS}, and is conjectured to be related to knot homology \cite{ORS}. Sometimes these relationships are discussed in terms of compactified Jacobians, which are equivalent to affine Springer fibers \cite{Laumon}.

Given a matrix $\gamma \in \gln((t))$ (or $\sln((t))$), the affine Springer fiber $\Sp_{\gamma}$ is a certain ind-subvariety of the affine Grassmanian, see Definition \ref{def: ASF}.

The connection to knot homology is of particular interest. Given $\gamma$, its characteristic polynomial $p(\lambda, t) = \text{det}(\gamma - \lambda I)$  defines a singular curve $C_\gamma$ in $\C^2$, called the spectral curve of $\gamma$. As long as $C_\gamma$ is reduced, $\Sp_{\gamma}$ depends only on its spectral curve $C_\gamma$, in particular on its completion at the origin $(C_\gamma,0)$. In this paper we will work with $\gamma$ with distinct eigenvalues so that $C_\gamma$ will be reduced.

\begin{fact}[\cite{Yun}]
    If $(C_\gamma,0)$ is irreducible, then $\Sp_{\gamma}$ is a projective variety, but if $(C_\gamma,0)$ is not irreducible, $\Sp_{\gamma}$ is an ind-variety with infinitely many irreducible components.
\end{fact}

Intersecting the $C_\gamma$ with a small sphere around the origin (where $C_\gamma$ is often singular) gives a link $L_{\gamma}$ in $S^3$. Each irreducible component of $(C_\gamma,0)$ corresponds to a component of the link, and the intersection numbers of irreducible components are the linking numbers of the corresponding link components. Any smooth components of the curve correspond to unknots. Oblomkov, Rasmussen, and Shende \cite{ORS} have conjectured that the homology of $\Sp_{\gamma}$ is closely related to the triply-graded Khovanov-Rozansky homology \cite{KR2,KhSoergel} (also called $\mathrm{HHH}$) of $L_\gamma$. This relationship has previously been shown for all torus knots, and for $(n, nd)$-torus links by Kivinen in \cite{Kivinen}.

\begin{conj}[ORS \cite{ORS}]
    \label{conj: ORS}
    If $L_{\gamma}$ is the link associated to $\gamma$, we have
    $$\mathrm{gr}_P H_*(\Sp_{\gamma}) \otimes \C[x] \cong \mathrm{HHH}^{a=0}(L_{\gamma})$$
    where $\mathrm{gr}_P$ is a certain perverse filtration on $H_*(\Sp_{\gamma})$.
\end{conj}

\begin{example}
    For the matrix
$$\gamma = \begin{pmatrix}
0 & t^2\\
t & 0
\end{pmatrix},$$
the characteristic polynomial is given by $p(\lambda, t) = \lambda^2-t^3.$ The associated link to this curve in $\C^2$ is a trefoil. The affine Springer fiber $\Sp_\gamma$ is isomorphic to $\PP^1$, and the reduced HHH homology of the trefoil is isomorphic to $H_*(\PP^1)$.

\end{example} 

In this paper, we will calculate the homology for a large class of affine Springer fibers with $\gamma$ of the form
\begin{equation}
\label{gamma def}
    \gamma = \begin{pmatrix}
         z_1t^{d_1} &  & 0\\
         & \ddots & \\
        0 & & z_nt^{d_n}
        \end{pmatrix},
\end{equation}
with $z_i \in \C^*$ pairwise distinct and $d_i > 0$, under the assumption that $\Sp_\gamma$ is equivariantly formal (see Definition \ref{def: eq formal}).

The characteristic polynomial of this $\gamma$ is $\prod_i (\lambda - z_i t^{d_i})$. The corresponding curve $C_\gamma$ has $n$ smooth irreducible components with pairwise intersection numbers $d_{ij} = \min(d_i, d_j)$. So the corresponding link $L_{\gamma}$ is a link of $n$ unknots with pairwise linking numbers $d_{ij}$. In Section \ref{sec: recursion}, we calculate $\HHH$ of this link $L_\gamma$ for $n=3$ and compare the result to $H_*(\Sp_\gamma)$. 

\subsection{Homology of Affine Springer Fibers}
\label{sec: intro ASF}

In Section \ref{sec: ASF}, we focus on computing the equivariant Borel-Moore homology $H_*^T(\Sp_\gamma)$ with respect to the natural torus action of $T = (\C^*)^n$ on $\Sp_{\gamma}$ (explained in Section \ref{sec: background}). Given the assumption that $\Sp_\gamma$ is equivariantly formal with respect to $T$, we can recover the ordinary Borel-Moore homology of $\Sp_\gamma$ by quotienting out by the equivariant parameters, see Fact \ref{fact: mod y}.

In order to calculate $H_*^T(\Sp_{\gamma})$, we view it as a module over 
$$R = H^*_T(pt) \otimes \C[\Lambda] \cong \C[t_1, \dots, t_n, x_1^{\pm}, \dots, x_n^{\pm}].$$
Here, the $t_i$'s are our equivariant parameters, and the $x_i$'s parametrize the integer lattice $\Lambda$ which acts on $\Sp_{\gamma}$ by translations.

In \cite{GKM unramified}, Goresky, Kottwitz, and MacPherson (GKM) conjectured that $\Sp_{\gamma}$ is pure for all unramified (i.e. diagonal) $\gamma$. The following is a more narrow conjecture that is what we will rely on in this paper.

\begin{conj}[\cite{GKM unramified}]
    For $\gamma$ as above, $\Sp_{\gamma}$ is equivariantly formal, as defined in Definition \ref{def: eq formal}.
\end{conj}

Assuming that $\Sp_{\gamma}$ is equivariantly formal, we can calculate its equivariant Borel-Moore homology.
\begin{thm}
\label{thm: intro J}   
    Consider $\gamma$ as in (\ref{gamma def}). Define the ideal
        
        $$\mathcal{J} = \mathcal{J}_{\gamma} = \bigcap_{i < j} (t_i-t_j, x_i - x_j)^{d_{ij}} \subseteq R$$

    with $d_{ij} = \text{min}(d_i, d_j)$. If $\Sp_{\gamma}$ is equivariantly formal, then as $R$-modules,
        
        $$\Delta H_*^T(\Sp_{\gamma}) \cong \cJ\quad \mathrm{where}\quad
        \Delta = \prod_{i < j} (t_i-t_j)^{d_{ij}}.$$
        
        \end{thm}
These $\cJ$ ideals are slightly more general than the similar ideals considered by Haiman in his work on the Hilbert schemes of points \cite{Haiman}, so we refer to them as generalized Haiman ideals. Since $\Delta$ is a polynomial, multiplication by $\Delta$ is injective, and 
$$H_*^T(\Sp_{\gamma}) \cong \Delta H_*^T(\Sp_{\gamma}) \cong \cJ$$
as modules over $R = \C[t_1, \dots, t_n, \x_1^{\pm}, \dots, x_n^{\pm}]$. It is still useful to keep track of $\Delta$ if we want to retain the localization information of $H_*^T(\Sp_{\gamma})$, but we can omit $\Delta$ when we only care about $H_*^T(\Sp_{\gamma})$ as an $R$-module. Given the assumption that $\Sp_{\gamma}$ is equivariantly formal, we can recover the ordinary Borel-Moore homology of $\Sp_{\gamma}$ as well.

\begin{cor}
    \label{cor: intro J/yJ}
    For $\gamma$ as in \eqref{gamma def}, if $\Sp_{\gamma}$ is equivariantly formal, then
    $$H_*(\Sp_{\gamma}) \cong \cJ/(\mathbf{t})\cJ.$$
    Here $(\textbf{t}) \subseteq H_T^*(pt) \cong \C[\mathbf{t}]$ is the maximal ideal generated by $t_1, \dots, t_n$.
\end{cor}

If $n=3,4$, it is known that $\Sp_{\gamma}$ is equivariantly formal, shown in \cite{n=3} and \cite{n=4} respectively. It is also known for the equivalued case, when $d_i=d$ for all $i$, due to GKM \cite{GKM unramified}.

\begin{cor}
\label{cor: known eq formal cases}
If $n \leq 4$, or if $d_i = d$ for all $i$, then
$$\Delta H_*^T(\Sp_{\gamma}) \cong \cJ = \bigcap_{i < j} (t_i-t_j, x_i - x_j)^{d_{ij}}.$$
\end{cor}

The equivalued case of \ref{cor: known eq formal cases} was previously shown by Kivinen in \cite{Kivinen} using GKM theory as defined in \cite{GKM}. The proof of Theorem \ref{thm: intro J} relies on this result by Kivinen.

In \cite{sheaf} Gorsky, Kivinen and Oblomkov define a graded algebra $\mathcal{A}_G = \bigoplus_{d=0}^\infty {}_0 \mathcal{A}_d$ depending only on a reductive group $G$ with some specific properties, called the graded Coulomb branch algebra. Here we consider the case $G = \Gln$. One of the key properties is that for any $\gamma \in \mathfrak{g},$ the direct sum of homologies of affine Springer fibers

$$F_{\gamma} = \bigoplus_{k=0}^{\infty} H_*(\Sp_{t^k \gamma})
$$
is a graded module over $\mathcal{A}_G$ or, equivalently, that there is a corresponding quasi-coherent sheaf $\mathcal{F}_{\gamma}$ on $\Proj\ \bigoplus_{d=0}^\infty {}_0 \mathcal{A}_d$. They conjecture the following.

\begin{conj}[\cite{sheaf}]
    \label{intro conj: coulomb}
    The module $F_{\gamma}$ is finitely generated and the sheaf $\mathcal{F}_{\gamma}$ is coherent.
\end{conj}

\begin{thm}
    Conjecture \ref{intro conj: coulomb} holds for $G = \mathrm{GL}_3$ and $\gamma$ as in (\ref{gamma def}).
\end{thm}

This result and many of the results in Section \ref{sec: intro J} rely on the specific combinatorics of the ideal $\cJ$ when $n=3$, which is covered in detail in Section \ref{sec: n=3}, with additional proofs in Section \ref{sec: proofs}.

\subsection{Generalized Haiman Ideals}
\label{sec: intro J}

For the rest of the introduction we assume that the $d_i$'s are ordered: $d_1\leq \ldots\leq d_n$. We will 
consider a similar ideal to $\cJ$ defined above,
$$
J'(d_1,\ldots,d_n)=\bigcap_{i < j} (t_i-t_j, x_i - x_j)^{d_{ij}} \subseteq \C[t_1,\ldots,t_n,x_1,\ldots,x_n].
$$
The ideal $\cJ$ is obtained from $J'(d_1,\ldots,d_n)$ by localization in $(x_1\cdots x_n)$.

In Section \ref{sec: Hilbert}, we define two rational functions, $H(d_1,\ldots,d_n)$ and $F(d_1,\ldots,d_n)$. The function $F(d_1, \dots, d_n)$ is also known as the generalized $(q,t)$-Catalan number, see \cite{qt}.

\begin{conj}
\label{intro conj: Hilbert}
a) The Hilbert series of the ideal $J'(d_1,\ldots,d_n)$ equals $H(d_1,\ldots,d_n)$.

b) The Hilbert series of the generating set $J'(d_1,\ldots,d_n)/\mathfrak{m}J'(d_1,\ldots,d_n)$ equals $F(d_1,\ldots,d_n)$, where $\mathfrak{m}$ is the maximal ideal $\mathfrak{m}=(t_1,\ldots,t_n,x_1,\ldots,x_n)$.
\end{conj}

In particular, Conjecture \ref{intro conj: Hilbert} implies  that $F(d_1,\ldots,d_n)$ is a polynomial in $q$ and $t$ with nonnegative coefficients (see \cite[Conjecture 1.3]{qt}) and provides an explicit algebraic interpretation of these coefficients. Similarly, the conjecture implies that $H(d_1,\ldots,d_n)$ is a power series in $q$ and $t$ with nonnegative coefficients. In Theorem \ref{thm: Hilbert}, we show that this Conjecture holds for $n=3$.

\begin{thm}
Conjecture \ref{intro conj: Hilbert} holds for $n=3$.    
\end{thm}

If $d_i = d$ for all $i$, we will say that the ideal $J'$ is equivalued. In \cite{Haiman}, Haiman shows the following.

\begin{thm}[Haiman \cite{Haiman}]
\label{thm: Haiman}
    For any $n$,
    \begin{enumerate}
        \item The ideal $J'(d, \ldots, d)$ is free as a $\C[\mathbf{t}]$-module.
        \item  The ideal $J'(d, \ldots, d)$ is equal to a product,
    $$J'(d, \ldots, d) = J'(1, \ldots, 1)^d.$$
    \end{enumerate}
   
\end{thm}

It is easy to see that in general 
$$J'(d_1, \dots, d_n) \cdot J'(d_1', \dots d_n') \subseteq J'(d_1 + d_1', \dots, d_n + d_n').$$ We conjecture that  Theorem \ref{thm: Haiman} can be generalized to the non-equivalued case, and that the above inclusion is always an equality.
\begin{conj}
    \label{conj: J product}
    For any $n$, assume $d_1\le d_2\le \cdots \le d_n$. Then,
    \begin{enumerate}
        \item The ideal $J'(d_1, \ldots, d_n)$ is free as a $\C[\mathbf{t}]$-module.
        \item The ideal $J'(d_1, \ldots, d_n)$ can be written as the product
    $$J'(d_1, \ldots, d_n) = J'(1, \ldots, 1)^{d_1} \cdot J'(0, 1, \ldots, 1)^{d_2-d_1} \cdot \ldots \cdot J'(0, \ldots, 0, 1)^{d_{n-1} - d_{n-2}}.$$
    \end{enumerate}
\end{conj}

Statement (1) immediately follows in any cases where $\Sp_{\gamma}$ is known to be equivariantly formal, in particular for $n \leq 4$. In Corollary \ref{cor: J111} we show that statement (2) holds in the $n=3$ case.

\begin{thm}
    \label{intro: J111}
    The ideal $J'(d_1, d_2)$ can be written as a product 
    $$J'(d_1, d_2) = J'(1,1)^{d_1} \cdot J'(0,1)^{d_2 - d_1}$$  
\end{thm}

\subsection{Link Homology}
\label{sec: intro link hom}
In Section \ref{sec: recursion}, we use a recursive process of Hogancamp and Elias \cite{torus links} to compute $\HHH^{a=0}(L_{\gamma})$ for $n=3$. We are then able to relate $\HHH^{a=0}(L_{\gamma})$ to the ideal $\cJ$, and thus the homology of $\Sp_{\gamma}$, to show that a weaker version of Conjecture \ref{conj: ORS} (without the perverse filtration) holds in this case. In order to show this, we use the y-ified Khovanov-Rozansky homology $\HY$ as defined by Gorsky and Hogancamp in \cite{GH}. 

\begin{thm}
\label{thm: intro HHH}
    For $n=3$ and $\gamma$ as in $\ref{gamma def}$,
    $$\HY^{a=0}(L_{\gamma})\otimes_{\C[\mathbf{x}]}\C[\mathbf{x},\mathbf{x}^{\pm}] \cong \Delta H_*^T(\Sp_{\gamma})$$
and $$\HHH^{a=0}(L_{\gamma})\otimes_{\C[\mathbf{x}]}\C[\mathbf{x},\mathbf{x}^{\pm}] \cong H_*(\Sp_{\gamma}).$$
\end{thm}

We are optimistic that the proof of Theorem \ref{thm: intro HHH} can be generalized to larger $n$, and one
may actually be able to find an explicit recursion to compute $\HHH^{a=0}(L_{\gamma})$ for any $n$, which would tell us more about the structure of $H_*(\Sp_{\gamma})$ beyond $n=3$. 

\subsection{The Fundamental Domain}

Finally, in the appendix  we discuss the cells of the fundamental domain of $\Sp_{\gamma}$, as described by Chen in \cite{Chen}, and relate this to the combinatorics of $J'$ for $n=3$. We show that there is a bijection between half of the cells in the fundamental domain and the generators of the ideal $J'(d_1,d_2).$ We expect that this bijection indicates a stronger relationship between the cells of $\Sp_\gamma$ and the combinatorics of $J'$ and of generalized $(q,t)$-Catalan numbers.

\section*{Acknowledgments}

This work was partially supported by the NSF grants
DMS-1760329  and DMS-2302305. The author would like to thank Erik Carlsson, Eugene Gorsky, Matt Hogancamp, Oscar Kivinen, Yuze Luan, and Alexei Oblomkov for useful discussions.

\section{Background}
\label{sec: background}
First, some notation. Let $\cK = \C((t))$ be the field of Laurent power series in $t$, and $\cO = \C[[t]]$ be the ring of power series in $t$. For nonzero $f \in \cK$, let $\nu(f)$ denote the order of $f$, which is the degree of the smallest nonzero term. Throughout, $H_*^T(X)$ will refer to the equivariant Borel-Moore homology of $X$, and $H_*(X)$ refers to regular Borel-Moore homology.

\begin{defn}
    A \textit{lattice} $\Lambda \subseteq\cK^n$ is a free $\cO$-submodule of $\cK^n$ of rank $n$ such that $\Lambda \otimes_{\cO} \cK = \cK^n$. In other words, it is the $\cO$-span of a basis of $\cK^n$ over $\cK$.
\end{defn}

\begin{defn}
 The \textit{affine Grassmanian} $\text{Gr}_{\text{GL}_n}(\C)$ of $\text{GL}_n$ over $\C$ is an ind-scheme defined as the space of all lattices $\Lambda \subseteq \mathcal{K}^n$.
\end{defn}

We can equivalently define the affine Grassmanian as $$\Gln(\cK)/\Gln(\cO),$$
as $\Gln(\cK)$ acts transitively on the space of lattices, and the stabilizer of the standard lattice $\cO^n$ is precisely $\text{GL}_n(\cO)$. We will often conflate a matrix $g \in \Gln(\cK)$ with its coset representative in $\Gr(\Gln)$.

We define $\Gr(\Sln)$ similarly, either as $\Sln(\cK)/\Sln(\cO)$, or as the space of lattices of $\Sln$ type. We say that a lattice $\Lambda$ is of $\Sln$ type if it can be written as $\Lambda = g\cO^n$ for some $g \in \Sln(\cK)$. 

We will always be working over $\C$, so we will use $\Gr(\Gln)$ or $\Gr(\Sln)$ for the affine Grassmanian, or just $\Gr$ when the group $G$ is clear.


\begin{defn}[\cite{KL}]
\label{def: ASF}
     The \textit{affine Springer fiber} $\Sp_{\gamma}$ of an element $\gamma \in \mathfrak{gl}_n(\mathcal{K})$ is a sub ind-scheme of the affine Grassmanian, defined accordingly as the space of lattices $\Lambda \in \text{Gr}_{\text{GL}_n}$ such that $\gamma \Lambda \subseteq\Lambda$, or as the space of $g \in \text{GL}_n(\mathcal{K})/\text{GL}_n(\mathcal{O})$ such that $g^{-1} \gamma g \in \mathfrak{gl}_n(\mathcal{O})$.
\end{defn}

\begin{fact}[\cite{Yun}]
    If $\gamma$ is regular, semi-simple, and topologically nilpotent, then $\Sp_{\gamma}$ is finite dimensional, although it can still have infinitely many irreducible components. In our cases, these conditions essentially are that $\gamma$ is diagonalizable, has distinct eigenvalues $\lambda_i$ (over $\bar{\mathcal{K}}$), and that $\nu(\lambda_i) > 0$ for all $i$.
\end{fact}

\begin{example}[\cite{Yun}]
    Consider the matrix
    $$\gamma = \begin{pmatrix}
    t & 0\\
    0 & -t
    \end{pmatrix}$$
    in $\mathfrak{sl}_2(\mathcal{K})$. Then $\Sp_{\gamma}$ looks like an infinite chain of $\PP^1$'s connected at $0$ and $\infty$. There is a $\C^*$ action that scales each $\PP^1$, and a $\Z$ action that translates them.
\end{example}

There is a natural action of the centralizer $C(\gamma)$ on $\Sp_{\gamma}$. In the case where $\gamma$ is diagonal, this gives a torus action of $T = (\C^*)^n $ and lattice action of $\Lambda = \Z^n$ on $\Sp_{\gamma}$ over $\Gln$ (or $(\C^*)^{n-1} $ and $\Z^{n-1}$ over $\Sln$) that can be respectively seen as multiplication by the matrices in $C(\gamma)$:
$$
\lambda = 
\begin{pmatrix}
    \lambda_1 & & \\
     & \ddots & \\
      & & \lambda_n
\end{pmatrix}$$
and 
$$Z = 
\begin{pmatrix}
    t^{m_1} & & \\
     & \ddots & \\
     & & t^{m_n}
\end{pmatrix}.$$

In general for a $T$-ind-scheme $X$, $H_*^T(X)$ is naturally a module over $H_T^*(pt)$ via the cap product.

\begin{defn}
\label{def: eq formal}
    We will say that $X$ is \textit{equivariantly formal} if $H_*^T(X)$ is free as a module over $H_T^*(pt)$.
\end{defn} 

\begin{fact} (GKM \cite{GKM})
    \label{fact: mod y}
    If $X$ is equivariantly formal, then 
    $$H_*(X) \cong H_*^T(X)/(\mathbf{t}) H_*^T(X),$$

    as modules over $H_T^*(pt)$. Here $(\textbf{t}) \subseteq H_T^*(pt) \cong \C[\mathbf{t}]$ is the maximal ideal generated by $t_1, \dots, t_n$.
\end{fact}

We will need the following localization lemma as stated by Brion.

\begin{lemma}[Brion \cite{Brion}]
\label{brion}
    Let $X$ be a $T$-ind-scheme and $T' \subseteq T$ a subtorus. If $i: X^{T'} \rightarrow X$ is the inclusion of $T'$-fixed points, then the induced map $$
    i_*: H^T_*(X^{T'}) \rightarrow H^T_*(X)
    $$
    is an isomorphism after inverting finitely many characters of $T$ that restrict nontrivially to $T'$. Further, if $X$ is equivariantly formal, then the induced map $i_{*}$ is injective, and we have that
    $$\prod_{\chi, \chi_{T'} \neq 0} \chi H_*^T(X) \subseteq H_*^T(X^{T'})$$
    where we take the product over all characters of $T$ that restrict nontrivially to $T'$.
\end{lemma}

We will also make frequent use of the Iwasawa decomposition for $\Gr(\Gln)$, which tells us that all $g\in \Gr(\Gln)$ can be represented by a product $DU$ of a diagonal matrix
\begin{equation}
\label{defn D}
D = 
\begin{pmatrix}
    t^{m_1} & & \\
    & \ddots & \\
    & & t^{m_n} 
\end{pmatrix}
\end{equation}
with a unipotent matrix $U$ with $1$'s on the diagonal and entries $\chi_{ij}$ above the diagonal \cite{Yun}. 
Further, these $\chi_{ij}$'s are unique up to $\cO$, so we can choose them to have all coefficients of nonnegative powers of  $t$ be $0$, so that each matrix $DU$ represents a unique element $g \in \Gr$.

\begin{lemma}
\label{Gr fixed pts}
    The $T$-fixed points of $\Gr(\Gln)$ can be uniquely represented by diagonal matrices $D$ as in the Iwasawa decomposition.
\end{lemma}

    \begin{proof}
        Let $\lambda \in T$ and $g = DU$ be as in the Iwasawa decomposition. Since $\lambda^{-1} \in \Gln(\cO)$, up to multiplication on the right by $\Gln(\cO),$ we get
        $$\lambda g = \lambda g \lambda^{-1} = DU',$$
        where $D$ is as above, and $U'$ is unipotent with $\frac{\lambda_i}{\lambda_j}\chi_{ij}$ above the diagonal. If $g$ is a fixed point under the action of $T$, we must have $\frac{\lambda_i}{\lambda_j}\chi_{ij} = \chi_{ij}$ for all $i,j$ and for all $\lambda \in T$. This can only happen if $\chi_{ij} = 0$ for all $i,j$, so $g = D$ is diagonal as desired.
    \end{proof}

Since the $T$-action on $\Sp_{\gamma}$ comes from the action on $\Gr$, the $T$-fixed points of $\Sp_{\gamma}$ are simply the $T$-fixed points of $\Gr$ that are contained in $\Sp_{\gamma}$. 

\begin{lemma}
\label{gamma fixed pts}
    If $\gamma$ is diagonal and the orders of its eigenvalues are all nonnegative, then
    $$\Sp_{\gamma}^T = \Gr^T.$$
\end{lemma}
\begin{proof}
    For any $\gamma$, $\Sp_{\gamma}^T \subseteq \Gr^T$ as stated above. If $g \in \Gr^T$, then
    $$g =  
\begin{pmatrix}
    t^{m_1} & & \\
    & \ddots & \\
    & & t^{m_n} 
\end{pmatrix}$$
by Lemma \ref{Gr fixed pts}. As $g$ and $\gamma$ are diagonal, $g^{-1}\gamma g = \gamma,$ and $\gamma \in \mathfrak{gl}_n(\cO)$, since the eigenvalues of $\gamma$ are all in $\cO$. So $\Gr^T \subseteq \Sp_{\gamma}^T$. 
\end{proof}

In particular, this means that the $T$-fixed points of $\Sp_{\gamma}$ are discrete and isomorphic to the integer lattice $\Lambda = \Z^n$, so can view $H_*(\Sp_{\gamma})$ as a module over 
$$H^*_T(\Sp_{\gamma}^T) \cong H^*_T(pt) \otimes \C[\Lambda] \cong \C[t_1, \dots, t_n, x_1^{\pm}, \dots, x_n^{\pm}].$$
Here, the $t_i$'s are our equivariant parameters, and a monomial $x_1^{a_1} \cdots x_n^{a_n}$ corresponds to the fixed point $\diag(t^{a_1},\ldots,t^{a_n})$.
The lattice $\Lambda$ acts on $\Sp_{\gamma}^T$ and on $\Sp_{\gamma}$ by translation.

\begin{lemma}
\label{lem: T' fixed Gr}
    Fix $i < j$. If $T' \subseteq T$ is a codimension 1 subtorus cut out by $t_i = t_j$, then the $T'$-fixed points of $\Gr(\Gln)$ are of the form $DU$, where $D$ is as in \eqref{defn D} and
    $$U = 
    \begin{pmatrix}
    1 & & \chi_{ij}\\
    & \ddots & \\
    & & 1 
    \end{pmatrix}$$
 with all $\chi$'s zero except for $\chi_{ij}$.
\end{lemma}
    \begin{proof}
        As before, up to equivalence,

        $$\lambda g = \lambda g \lambda^{-1} = DU',$$
 where $U'$ has $\frac{\lambda_k}{\lambda_l}\chi_{kl}$ above the diagonal. If $g$ is a fixed point for $T$, since $\lambda_i = \lambda_j,$ $\chi_{ij}$ can be arbitrary, but $\chi_{kl}=0$ for all $(k,l) \neq (i,j)$. So the fixed points are as described.
    \end{proof}

\begin{cor} If $T'$ is a codimension 1 subtorus of $T$ cut out by $t_i=t_j$, then
    $$\Gr(\Gln)^{T'} \cong \Gr(\mathrm{GL}_2) \times \Z^{n-2}.$$
\end{cor}
\begin{proof}
Each of the $T'$ fixed points is represented by $DU$ above. Looking at the $2\times 2$ submatrix of $DU$ in rows and columns $i,j$, we see a copy of $\Gr(\mathrm{GL}_2)$. The rest of the $m_i$ are free integers, and there are $n-2$ of them.
\end{proof}

\section{Borel-Moore Homology of \texorpdfstring{$\Sp_{\gamma}$}{Spg}}
\label{sec: ASF}
We want to find the equivariant Borel-Moore homology $H_*^T(\Sp_\gamma)$ of the class of affine Springer fibers with
$$
\gamma = 
        \begin{pmatrix}
         \gamma_1 &  & 0\\
         & \ddots & \\
        0 & & \gamma_n
        \end{pmatrix}
        =
        \begin{pmatrix}
         z_1t^{d_1} &  & 0\\
         & \ddots & \\
        0 & & z_nt^{d_n}
        \end{pmatrix}.
$$
Here $z_i \in \C^*$ are pairwise distinct and $d_i \geq 0$. We can assume up to a change of basis that $d_1 \leq \dots \leq d_n$.

It is known that for $\gamma$ as above, $\Sp_{\gamma}$ is equivariantly formal (see Definition \ref{def: eq formal}) over $\text{GL}_n$ for $n\leq 4$ (see \cite{n=3} for $n=3$ and \cite{n=4} for $n=4$), and it is known to be equivariantly formal for all $n$ if $d_i = d$ for all $i$ \cite{GKM purity}. But it is not known over $\text{GL}_n$ in general. It would be sufficient to know that the homology of $\Sp_{\gamma}$ is supported in even degrees, and we conjecture that this is the case for all $n$. We will need to assume that $\Sp_{\gamma}$ is equivariantly formal in order to calculate its homology.

We consider $H_*^T(\Sp_\gamma)$ as a module over $$H^*_T(\text{pt})\otimes \C[\Z^n] \cong \C[t_1, \dots, t_n, x_1^{\pm}, \dots, x_n^{\pm}] = R.$$

\begin{thm}
\label{thm: J}
    Consider $\gamma$ as in \eqref{gamma def}. Define the ideal 
         $$\cJ \subseteq R$$
        $$\cJ =  \bigcap_{i < j} (t_i-t_j, x_i - x_j)^{d_{ij}}$$

    with $d_{ij} = \text{min}(d_i, d_j)$. If $\Sp_{\gamma}$ is equivariantly formal, then as $R$-modules,   
        $$\Delta H_*^T(\Sp_{\gamma}) \cong \cJ,$$
        where $\Delta = \prod_{i < j} (t_i-t_j)^{d_{ij}}.$
    \end{thm}
Note that multiplication by $\Delta$ is injective, so 
$$H_*^T(\Sp_{\gamma}) \cong \Delta H_*^T(\Sp_{\gamma}) \cong \cJ$$
as $R$-modules. It can be useful to keep track of $\Delta$ if we want to retain the localization information of $H_*^T(\Sp_{\gamma})$, but we can omit $\Delta$ when we only care about $H_*^T(\Sp_{\gamma})$ as an $R$-module.

The rough outline of the proof of Theorem \ref{thm: J} is as follows:

\begin{enumerate}
    \item Take a codimension one subtorus $T' \subseteq T$. The $T'$-fixed points of $\Sp_\gamma$ are essentially isomorphic to an affine Springer fiber $\Sp_{\Tilde{\beta}}$ with $\Tilde{\beta} \in \mathfrak{gl}_2$ whose homology is known.
    
    \item Relate the homology of $\Sp_{\Tilde{\beta}}$ to that of $\Sp_\gamma$ using Lemma \ref{brion}.
    
    \item Take enough subtori $T'$ and piece together their homologies to find the homology of $\Sp_{\gamma}$.

\end{enumerate}

Step 3 will require the assumption that $\Sp_{\gamma}$ is equivariantly formal. 
\begin{lemma}
\label{lem: T' fixed points}
    If $T' \subseteq T$ is the subtorus cut out by $t_i=t_j$, then up to $\Z^{n-2}$, the $T'$-fixed points of $\Sp_{\gamma}$ are isomorphic to an affine Springer fiber over $\text{GL}_2$,
    $$
    \Sp_{\gamma}^{T'} \cong \Sp_{\beta_{ij}} \times \Z^{n-2},
    $$         
    where   
    $$
    \beta_{ij} = 
    \begin{pmatrix}
    z_i t^{d_i} & 0\\
    0 & z_j t^{d_j}
    \end{pmatrix}.
    $$ 
\end{lemma}

\begin{proof}
In Lemma \ref{lem: T' fixed Gr} we've already characterized the $T'$-fixed points of $\Gr$ as $DU$, where $U$ is unipotent with only a single nonzero $\chi_{ij}$. We just need to check which of those fixed points are in $\Sp_{\gamma}$. If $g \in \Gr(\Gln)^{T'}$, then

$$g^{-1} \gamma g = 
\begin{pmatrix}
    z_1t^{d_1} &  & \chi_{ij}(z_i t^{d_i} - z_j t^{d_j})\\
    & \ddots & \\
    0 & & z_nt^{d_n}
\end{pmatrix}.
$$
Again looking at the $2\times 2$ $i,j$ submatrix, we see a matrix identical to $g^{-1} \beta_{ij} g$. So a $T'$-fixed point $g$ is in $\Sp_{\gamma}$ if and only if the $2\times 2$ matrix
$$\begin{pmatrix}
    t^{m_i} & \chi_{ij}\\
    0 & t^{m_j}
\end{pmatrix}$$
is in $\Sp_{\beta_{ij}}$, i.e.
$
\Sp_{\gamma} \cong \Sp_{\beta_{ij}} \times \Z^{n-2}.
$
\end{proof}


\begin{lemma}
\label{lem: beta tilde}
    Given $\beta_{ij} \in \mathfrak{gl}_2$ as in Theorem 1, we have $\Sp_{\beta_{ij}} \cong \Sp_{\tilde{\beta}_{ij}},$ where 
    
    $$\tilde{\beta}_{ij} = \begin{pmatrix}
    z_i t^{d_{ij}} & 0\\
    0 & z_j t^{d_{ij}}
    \end{pmatrix}
    $$
and $d_{ij} = \text{min}(d_i, d_j)$.
\end{lemma}

\begin{proof}
    Again using the Iwasawa decomposition, write $g = DU$, where 
    $$U = 
    \begin{pmatrix}
        1 & \chi_{ij}\\
        0 & 1
    \end{pmatrix}.$$ 
    Then,
    $$g^{-1}\gamma g = 
    \begin{pmatrix}
        \gamma_i & \chi_{ij}(\gamma_i - \gamma_j)\\
        0 & \gamma_j
    \end{pmatrix}.
    $$
By definition, $g \in \Sp_{\beta_{ij}}$ if and only if $\chi_{ij}(\gamma_i - \gamma_j) \in \cO$. Since we assume that the $z_i$ are distinct, $\nu(\gamma_i - \gamma_j) = \min(\nu(\gamma_i), \nu(\gamma_j)) = \min(d_i, d_j) = d_{ij}$. So $g \in \Sp_{\beta_{ij}}$ if and only if $\chi_{ij}$ has order at least $-d_{ij}$. This is the same as the condition for $g$ to be in $\Sp_{\Tilde{\beta}_{ij}}$, since
    $$g^{-1} \Tilde{\beta}_{ij} g = 
    \begin{pmatrix}
        z_i t^{d_{ij}} & (z_i-z_j)\chi_{ij} t^{d_{ij}}\\
        0 & z_j t^{d_{ij}}
    \end{pmatrix}.$$
\end{proof}

\begin{rem}
\label{rem: beta tilde equivariant}
The one-dimensional quotient torus $T/T'$ naturally acts on $\Sp_{\gamma}^{T'}$. On the other hand, $T/T'$ is isomorphic to the one-dimensional torus $(\C^*)^2/\C^*$ which acts on $\Sp_{\beta_{ij}}$ and on $\Sp_{\tilde{\beta}_{ij}}$.
The isomorphisms
constructed in Lemmas \ref{lem: T' fixed points} and \ref{lem: beta tilde} are $T/T'$-equivariant.
\end{rem}

The homology for the affine Springer fiber of matrices like $\tilde{\beta}_{ij}$ with all powers the same is known. It is found using GKM theory by Kivinen in \cite{Kivinen}.

\begin{thm}[Kivinen \cite{Kivinen}]
\label{thm: Kivinen}
    If
    $$\gamma = 
        \begin{pmatrix}
         z_1t^{d} &  & 0\\
         & \ddots & \\
        0 & & z_nt^{d}
        \end{pmatrix},$$
    then $\Delta H_*^T(\Sp_{\gamma})$ injects into 
    $$H^*_T(\text{pt}) \otimes \C[\Z^n] \cong \C[t_1, \dots, t_n, x_1^{\pm}, \dots, x_n^{\pm}],$$
    where 
    $\Delta = \prod_{i < j} (t_i-t_j)^d.$ As a submodule, there is a canonical isomorphism 
    $$\Delta H_*^T(\Sp_{\gamma}) \cong \prod_{i < j} (t_i-t_j, x_i-x_j)^{d}.$$
\end{thm}

\begin{cor}
\label{cor: ideal n=2}
    We have the following canonical isomorphism of $\C[t_i-t_j, x_i^{\pm}, x_j^{\pm}]$-modules: 
    $$(t_i-t_j)^{d_{ij}} H_*^{T/T'}(\Sp_{\tilde{\beta}_{ij}}) \cong (t_i-t_j, x_i-x_j)^{d_{ij}}\subseteq\C[t_i-t_j, x_i^{\pm}, x_j^{\pm}].$$
Here $T/T'$ acts on $\Sp_{\tilde{\beta}_{ij}}$  as in Remark \ref{rem: beta tilde equivariant}.
\end{cor}

 Now in order to piece together these homologies of $\Sp_{\tilde{\beta}_{ij}}$, we use the following fact:

\begin{lemma}[Algebraic Hartogs' lemma]
\label{lem: Hartogs}
    If $A$ is an integrally closed Noetherian integral domain, then

        $$A = \bigcap_{p \text{ codimension 1}} A_p,$$
        where we take the intersection over all codimension 1 prime ideals of $A$ of $A_p$ inside the fraction field Frac($A$).
    If $M$ is a free $A$-module, then we also have
    $$M = \bigcap_{p \text{ codimension 1}} M_p,$$
    where the intersection is taken inside $M \otimes_A \text{Frac}(A)$.
\end{lemma}

Since we assumed that $\Sp_{\gamma}$ is equivariantly formal, $H^T_*(\Sp_{\gamma})$ is free over $H^*_T(pt)$. Now we can prove Theorem \ref{thm: J}.

\begin{proof}[Proof of Theorem \ref{thm: J}]
    By Lemma \ref{brion}, we have that, up to localization away from $(t_i-t_j)$,
        
    $$H^T_*(\Sp_{\gamma}) \cong_{loc} H^T_*(\Sp_{\gamma}^{T'}) \cong H_*^{T/T'}\left(\Sp_{\gamma}^{T'}\right)\otimes H^*_{T'}(pt)= H^{T/T'}_*(\Sp_{\tilde{\beta}_{ij}} \times \Z^{n-2})\otimes H^*_{T'}(pt).$$

Here $\cong_{loc}$ indicates an isomorphism after localization. Note that after localization, $\Delta=(t_i-t_j)^{d_{ij}}$ up to an invertible factor, so by Lemma \ref{lem: beta tilde} we get
$$
\Delta H^T_*(\Sp_{\gamma}) \cong_{loc}(t_i-t_j)^{d_{ij}} H_*^{T/T'}\left(\Sp_{\widetilde{\beta}_{i,j}}\right)\otimes H^*_{T'}(pt)\otimes \C[\Z^{n-2}]=
(t_i-t_j, x_i-x_j)^{d_{ij}}\subseteq\C[\mathbf{t}, \mathbf{x}^{\pm}].
$$
We have inclusion map
$$
\Sp_{\gamma}^{T'}\xrightarrow{i} \Sp_{\gamma}
$$

and by Lemma \ref{brion}
$$\prod_{(k,\ell)\neq (i,j)}(t_k-t_{\ell})H_*^T(\Sp_{\gamma})\subseteq i_*H_*^T(\Sp_{\gamma}^{T'}).
$$

We also have that
$$
\Delta H_*^T(\Sp_{\gamma})\subseteq \prod_{(k,\ell)\neq (i,j)}(t_k-t_{\ell})(t_i-t_j)^{d_{ij}}H_*^T(\Sp_{\gamma}).
$$

By the above, this is contained in 
$$
i_{*}(t_i-t_j)^{d_{ij}}H_*^T(\Sp_{\gamma}^{T'})=
(t_i-t_j, x_i-x_j)^{d_{ij}}.
$$
We conclude that
$$\Delta H^T_*(\Sp_{\gamma}) \subseteq(t_i-t_j, x_i-x_j)^{d_{ij}}.$$


This holds for all codimension-1 subtori $T'=(t_i-t_j) \subseteq T$ with $i < j$, so we have

    $$\Delta H^T_*(\Sp_{\gamma}) \subseteq \bigcap_{i < j} (t_i-t_j, x_i-x_j)^{d_{ij}}.$$

In fact we have already seen that $(t_i-t_j, x_i-x_j)^{d_{ij}}$ is exactly the localization $H^T_*(\Sp_{\gamma})_{p}$ where $p = (t_i-t_j)$.

So by Lemma \ref{lem: Hartogs} we conclude that 

    $$\Delta H^T_*(\Sp_{\gamma}) \cong \bigcap_{i < j} (t_i-t_j, x_i-x_j)^{d_{ij}}.$$
\end{proof}

This proof required the assumption that $\Sp_{\gamma}$ is equivariantly formal. If $d_i=d$ for all $i$, then it is known to be equivariantly formal \cite{GKM purity} and we recover the homology result of Theorem \ref{thm: Kivinen} from Kivinen \cite{Kivinen}.

\begin{conj}
    $\Sp_{\gamma}$ is equivariantly formal for all $d_1,\ldots,d_n$ and all $n$.
\end{conj}


\section{Generalized Haiman Ideal \texorpdfstring{$J$}{J} for \texorpdfstring{$n=3$}{n=3}}
\label{sec: n=3}
When $n=3$, it is known that $\Sp_{\gamma}$ is equivariantly formal, so by Theorem \ref{thm: J}, up to denominators ($\Delta)$, its equivariant Borel-Moore homology is isomorphic to the ideal
$\mathcal{J} \subseteq \C[t_1, t_2, t_3, x_1^{\pm}, x_2^{\pm}, x_3^{\pm}]$ defined as
$$\mathcal{J}= \mathcal{J}(d_1, d_2) = (t_1-t_2,x_1-x_2)^{d_1} \cap (t_1-t_3,x_1-x_3)^{d_1} \cap (t_2-t_3,x_2-x_3)^{d_2}$$
seen as a module over $\C[t_1, t_2, t_3, x_1^{\pm}, x_2^{\pm}, x_3^{\pm}].$ Here we are assuming that $d_1 \leq d_2 \leq d_3$, so that $d_1,d_1,d_2$ are equal to the pairwise minima $d_{ij}$.  

We consider a similar ideal $J' \subset\C[t_1, t_2, t_3, x_1, x_2, x_3]$
$$J' = J'(d_1, d_2) = (t_1-t_2,x_1-x_2)^{d_1} \cap (t_1-t_3,x_1-x_3)^{d_1} \cap (t_2-t_3,x_2-x_3)^{d_2}$$

It is easy to see that $\mathcal{J}=J'\otimes_{\C[\mathbf{x}]}\C[\mathbf{x}^{\pm}]$, so
the generators for $\mathcal{J}$ over $\C[\bf{t}, \mathbf{x}^{\pm}]$ will be the same as generators of $J'$ over the polynomial ring $\C[\mathbf{t}, \mathbf{x}]$. 
Next we will do a change of variables:
$$a = t_1-t_2,\ b = x_1-x_2,\ c = t_3-t_2,\ d = x_3-x_2$$
and consider the ideal
$$J=J(d_1,d_2) = (a,b)^{d_1} \cap (c,d)^{d_2} \cap (a-c, b-d)^{d_1}$$
over $R= \C[a,b,c,d]$.
Clearly, we get
$$
J'(d_1,d_2)=J(d_1,d_2)\otimes_{\C}\C[x_1+x_2+x_3,t_1+t_2+t_3]
,$$
so again all three ideals have the same generators up to this change of variables.

We can also consider these as bigraded ideals, where the $t_i$'s (or $a$ and $c$) have bidegree $q$, and the $x_i$'s ($b$ and $d$) have bidegree $t$.

We will frequently use the polynomial
$$
ad-bc=d(a-c)-c(b-d)\in (a,b)\cap (c,d)\cap (a-c,b-d).
$$

\begin{thm}
\label{thm: generators}
    The ideal $J = (a,b)^{d_1} \cap (c,d)^{d_2} \cap (a-c,b-d)^{d_1}$ over $R = \C[a,b,c,d]$ has the following families of generators $(0 \leq j \leq d_1)$:
    \begin{enumerate}
        \item $A_{i,j}=a^{d_1-j}c^{d_2-j}(a-c)^{i}(b-d)^{d_1-j-i}(ad-bc)^{j},1\le i\le d_1-j$. These generators have bidegree $q^{d_1+d_2-j+i}t^{d_1-i}$, and there are $d_1-j$ of these for a fixed $j$. They are characterized by $\deg_t<d_1$.

        \item $B_{i,j}=a^{d_1-j-i}b^{i}d^{d_2-j}(b-d)^{d_1-j}(ad-bc)^{j},1\le i\le d_1-j$. Such generators have bidegree $q^{d_1-i}t^{d_1+d_2-j+i}$, and there are $d_1-j$ of these for a fixed $j$. They are characterized by $\deg_q<d_1$.

        \item $C_{i,j}=a^{d_1-j}c^{i}d^{d_2-j-i}(b-d)^{d_1-j}(ad-bc)^{j},1\le i\le d_2-j$. Such generators have bidegree $q^{d_1+i}t^{d_1+d_2-j-i}$, and there are $d_2-j$ of these for a fixed $j$. They are characterized by $\deg_q > d_1, \deg_t \geq d_1$.

        \item $D_j=a^{d_1-j}d^{d_2-j}(b-d)^{d_1-j}(ad-bc)^{j}$ has bidegree $q^{d_1}t^{d_1+d_2-j}$, there is one such generator for each $j$. They are characterized by $\deg_q=d_1$.
    \end{enumerate}
\end{thm}

\begin{rem}
\label{rem: j=d1}
For $j=d_1$, the generators $A_{i,j}$ and $B_{i,j}$ are not defined, while $C_{i,d_1}=c^id^{d_2-d_1-i}(ad-bc)^{d_1}$ for $1\le i\le d_2-d_1$ and $D_{d_1}=d^{d_2-d_1}(ad-bc)^{d_1}.$
\end{rem}

In particular, it is easy to see that there is at most one generator in each $(q,t)$-bidegree, see also Proposition \ref{prop: bijection}. Also notice that we chose the generators in Theorem \ref{thm: generators} such that the monomial factor in $A_{i,j}$ does not contain $b$ or $d$, and the monomial factors in $B_{i,j},C_{i,j},D_j$ do not contain $bc$ (unless $j=d_1$).

Theorem \ref{thm: generators} follows from Proposition \ref{prop: basis}, which we prove in Section \ref{proof: generators}. 

\begin{prop}
\label{prop: basis}
The ideal $J(d_1, d_2)$ has the following basis (over $\mathbb{C}$):
$$
m(a,c)A_{i,j}, m(a,b,d)B_{i,j}, m(a,c,d)C_{i,j}, m(a,d)D_k\ (j\leq d_1-1)
$$
where $m$ are arbitrary monomials in the corresponding variables. For $j=d_1$ we have to add all polynomials of the form
$$
a^{\alpha}b^{\beta}c^{\gamma}d^{\delta}(ad-bc)^{d_1},\ \gamma+\delta\geq d_2-d_1.
$$
\end{prop}

\begin{example}
\label{ex: 1 1 1}
For $d_1=d_2=1$ we get the following 5 generators of $J(1,1)$:
$$
A_{1,0}=ac(a-c),\ B_{1,0}=bd(b-d),\ C_{1,0}=ac(b-d),\ D_{0}=ad(b-d),\ D_1=(ad-bc).
$$
We can change the variables back to see that the generators of $\cJ$ over $\C[t_1, t_2, t_3, x_1^{\pm}, x_2^{\pm}, x_3^{\pm}]$ are
$$A_{1,0}=(t_1-t_2)(t_3-t_2)(t_1-t_3), \ B_{1,0}=(x_1-x_2)(x_3-x_2)(x_1-x_3),$$
$$C_{1,0}=(t_1-t_2)(t_3-t_2)(x_1-x_3), \ D_{0} =(t_1-t_2)(x_3-x_2)(x_1-x_3),$$
and
$$
D_1=\det\left(
\begin{matrix}
    1 & 1 & 1\\
    x_1 & x_2 & x_3\\
    t_1 & t_2 & t_3\\
\end{matrix}
\right)
$$
\end{example}

\begin{cor}
\label{cor: J111}
We have that $$J(d_1,d_2)=J(1,1)^{d_1}\cdot J(0,1)^{d_2-d_1},$$ $$J'(d_1,d_2)=J'(1,1)^{d_1}\cdot J'(0,1)^{d_2-d_1}$$ and $$\mathcal{J}(d_1,d_2)=\mathcal{J}(1,1)^{d_1}\cdot \mathcal{J}(0,1)^{d_2-d_1}.$$
\end{cor}

\begin{proof}
We prove the first equation, and the other two equations follow immediately.

The containment $J(1,1)^{d_1}\cdot J(0,1)^{d_2-d_1}\subseteq J(d_1,d_2)$ is clear, so it is sufficient to show that any generator of $J(d_1,d_2)$ can be written as a product of $d_1$ generators of $J(1,1)$ (listed in Example \ref{ex: 1 1 1}) and $d_2-d_1$ generators of $J(0,1)=(c,d)$. Indeed:
$$
A_{i,j}=a^{d_1-j}c^{d_2-j}(a-c)^{i}(b-d)^{d_1-j-i}(ad-bc)^{j}=(ac(a-c))^{i} \cdot (ac(b-d))^{d_1-j-i}\cdot (ad-bc)^{j}\cdot c^{d_2-d_1},
$$
$$
B_{i,j}=a^{d_1-j-i}b^{i}d^{d_2-j}(b-d)^{d_1-j}(ad-bc)^{j}=(ad(b-d))^{d_1-j-i}(bd(b-d))^{i}(ad-bc)^{j}d^{d_2-d_1},
$$
$$
C_{i,j}=a^{d_1-j}c^{i}d^{d_2-j-i}(b-d)^{d_1-j}(ad-bc)^{j}=
$$
$$(ac(b-d))^{x}(ad(b-d))^{d_1-j-x}(ad-bc)^{j}c^{i-x}d^{d_2-d_1-i+x},
$$
where $x=\max(0,i+d_1-d_2)$. Note that $i\le d_2-j$, so
$i+d_1-d_2\le d_1-j$ and hence $x\le d_1-j$. Also, $d_1\le d_2$, so $i+d_1-d_2\le i$ and $x\le i$. Therefore all exponents are indeed nonnegative. 
Finally,
$$
D_{j}=a^{d_1-j}d^{d_2-j}(b-d)^{d_1-j}(ad-bc)^{j}=(ad(b-d))^{d_1-j}(ad-bc)^{j}d^{d_2-d_1}.
$$
\end{proof}

\begin{rem}
Note that by Remark \ref{rem: j=d1} the polynomials
$a^{\alpha}b^{\beta}c^{\gamma}d^{\delta}(ad-bc)^{d_1},\ \gamma+\delta\geq d_2-d_1$ are either multiples of $C_{i,d_1}$ or of $D_{d_1}$.
\end{rem}

In \cite{sheaf}, Gorsky, Kivinen, and Oblomkov define a graded algebra $\mathcal{A}_G = \bigoplus_{d=0}^\infty  \mathcal{A}_d$, depending only on a reductive group $G$, with some specific properties. One of the key properties is that for any $\gamma \in \mathfrak{g},$ the direct sum of homologies of affine Springer fibers

$$F_{\gamma} = \bigoplus_{k=0}^{\infty} H_*(\Sp_{t^k \gamma})$$
is a graded module over $\mathcal{A}_G$, or equivalently, that there is a corresponding quasi-coherent sheaf $\mathcal{F}_{\gamma}$ on $\Proj \bigoplus_{d=0}^\infty  \mathcal{A}_d$. They conjecture that $F_{\gamma}$ is finitely generated and that this sheaf is coherent \cite[Conjecture 8.1]{sheaf}. In the case where $G= \Gln$, they show that this graded algebra is generated in degrees 0 and 1, and that $\Proj \bigoplus_{d=0}^\infty \mathcal{A}_d = \text{Hilb}^n(\C^* \times \C)$. A special case of this conjecture follows from Theorem \ref{thm: generators} and its corollaries.

 Indeed, it is proved in \cite{sheaf} that $\mathcal{A}_0$ is the space of symmetric polynomials in $\C[t_1,\ldots,t_n,x_1^{\pm},\ldots,x_n^{\pm}]$, and $\mathcal{A}_1$ is the space of antisymmetric polynomials. 
 

\begin{thm}
    In the case of $G = \mathrm{GL}_3$ and $\gamma$ as in (\ref{gamma def}), the graded module $F_{\gamma}$ is finitely generated over $\mathcal{A}_G$, and defines a coherent sheaf on $\mathrm{Hilb}^3(\C^* \times \C)$, i.e. Conjecture 8.1 holds in this case.
\end{thm}
\begin{proof}
    By Corollary 3.8.3 in \cite{Haiman}, the ideal generated by $\mathcal{A}_1$ is exactly $\cJ(1,1)$. There is a natural inclusion of ideals
    $$\mathcal{J}(d_1,  d_2)\cdot  \mathcal{J}(1,1) \rightarrow \mathcal{J}(d_1+1, d_2+1).$$ It follows from Corollary \ref{cor: J111} that this map is actually surjective as well. Since $H_*(\Sp_{t^k\gamma})$ corresponds to the ideal $\mathcal{J}(d_1+k,d_2+k)$, this shows that the module $F_{\gamma}$ is generated in degree 0, and therefore finitely generated by the generators of $\mathcal{J}(d_1,d_2)$. 
\end{proof}

\subsection{Hilbert Series}
\label{sec: Hilbert}

Let us introduce two rational functions
$$
H(d_1,\ldots,d_n)=\sum_{T}\frac{z_1^{d_n}\cdots z_n^{d_1}}{(1-q)^n(1-t)^n}\prod_{i=2}^{n}\frac{1}{(1-z_i^{-1})}\prod_{i<j}\omega(z_i/z_j)
$$
and
$$
F(d_1,\ldots,d_n)=\sum_{T}z_1^{d_n}\cdots z_n^{d_1}\prod_{i=2}^{n}\frac{1}{(1-z_i^{-1})(1-qtz_{i-1}/z_i)}\prod_{i<j}\omega(z_i/z_j).
$$
Here the sums are over standard tableaux $T$ with $n$ boxes, $z_i$ is the $(q,t)$-content $q^{c-1}t^{r-1}$ of  the box labeled by $i$ in row $r$ and column $c$ in $T$, and $\omega(x) = \frac{(1-x)(1-qtx)}{(1-qx)(1-tx)}$. By convention, all the factors of the form $(1-1)$ in the above products (either in the numerator or in denominator) should be ignored. 

The function $F(d_1,\ldots,d_n)$ is also known as generalized $(q,t)$-Catalan number, see \cite{qt} for more details and context. Note that the order of the $d_i$ is reversed here compared to \cite{qt}.

\begin{conj}
\label{conj: Hilbert}
We have that:

a) The Hilbert series of the ideal $J'(d_1,\ldots,d_n)$ equals $H(d_1,\ldots,d_n)$.

b) The Hilbert series of the generating set $J'(d_1,\ldots,d_n)/\mathfrak{m}J'(d_1,\ldots,d_n)$ equals $F(d_1,\ldots,d_n)$. Here $\mathfrak{m}$ is the maximal ideal $\mathfrak{m}=(t_1,\ldots,t_n,x_1,\ldots,x_n)$.
\end{conj}

In particular, this conjecture implies  that $F(d_1,\ldots,d_n)$ is a polynomial in $q$ and $t$ with nonnegative coefficients (see \cite[Conjecture 1.3]{qt}) and provides an explicit algebraic interpretation of these coefficients. 
Similarly, the conjecture implies that $H(d_1,\ldots,d_n)$ is a power series in $q$ and $t$ with nonnnegative coefficients. 

\begin{example}
For $n=2$ we get $J(d_1,d_2)=(x_1-x_2,y_1-y_2)^{d_1}$. We change coordinates  to $x_1-x_2=a,y_1-y_2=b$ and $\overline{x}=x_1+x_2,\overline{y}=y_1+y_2$, then $J(d_1,d_2)$ has generating set $a^{d_1},a^{d_1-1}b,\ldots,b^{d_1}$, so the Hilbert series for the generating set equals
$$
q^{d_1}+q^{d_1-1}t+\ldots+t^{d_1}=\frac{q^{d_1}}{1-t/q}+\frac{t^{d_1}}{1-q/t}=F(d_1,d_2).
$$
Similarly, $J(d_1,d_2)$ is free over $\C[\overline{x},\overline{y}]$ with basis $a^{\alpha}b^{\beta},\alpha+\beta\ge d_1$, so by Lemma \ref{lem: s} below we get the Hilbert series
$$
\frac{1}{(1-q)(1-t)}\sum_{\alpha+\beta\ge d_1}q^{\alpha}t^{\beta}=\frac{q^{d_1}}{(1-q)^2(1-t)(1-t/q)}+\frac{t^{d_1}}{(1-q)(1-t)^2(1-q/t)}.
$$
\end{example}

\begin{thm}
\label{thm: Hilbert}
Conjectures \ref{conj: Hilbert}(a) and \ref{conj: Hilbert}(b)  hold for $n=3$ for all $d_1,d_2,d_3$.
\end{thm}

The statement of (a) follows from the Hilbert series calculation in Section \ref{sec: Hilb series}, and the proof of (b) will be in Section \ref{sec: Comb of J}.

\subsection{Combinatorics of \texorpdfstring{$J$}{J}}
\label{sec: Comb of J}

We've already seen in Example \ref{ex: 1 1 1} that $J(1,1)$ has 5 generators, and that in general the generators of $J$ each have a unique $(q,t)$-bidegree. We can plot the bidegrees of these generators as below.
\begin{example}
Here is an example where $d_1=d_2=3$:

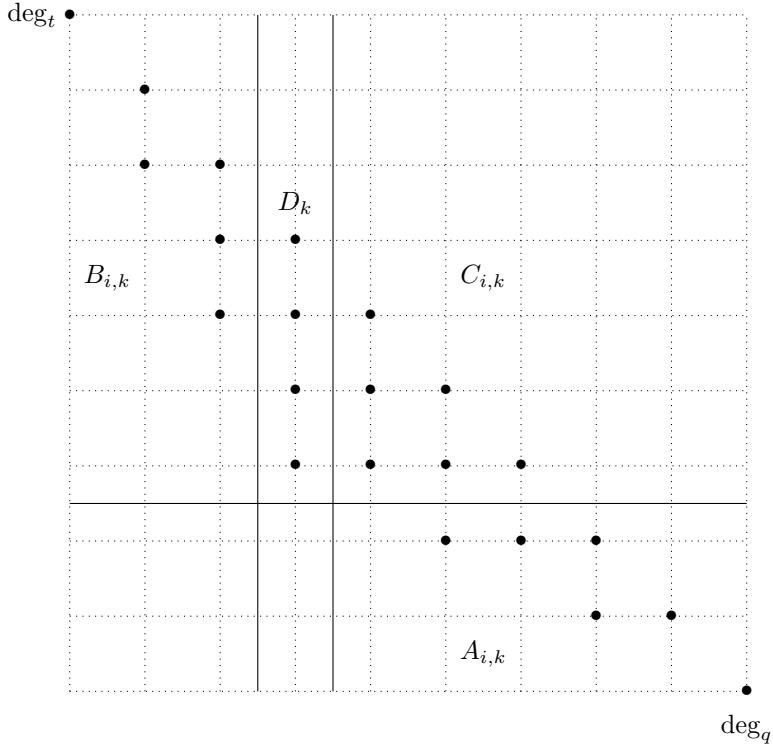
\begin{center}
\begin{tikzpicture}
\draw  [dotted] (0,0)--(0,9);
\draw  [dotted] (1,0)--(1,9);
\draw  [dotted] (2,0)--(2,9);
\draw  [dotted] (3,0)--(3,9);
\draw  [dotted] (4,0)--(4,9);
\draw  [dotted] (5,0)--(5,9);
\draw  [dotted] (6,0)--(6,9);
\draw  [dotted] (7,0)--(7,9);
\draw  [dotted] (8,0)--(8,9);
\draw  [dotted] (9,0)--(9,9);
\draw  [dotted] (0,0)--(9,0);
\draw  [dotted] (0,1)--(9,1);
\draw  [dotted] (0,2)--(9,2);
\draw  [dotted] (0,3)--(9,3);
\draw  [dotted] (0,4)--(9,4);
\draw  [dotted] (0,5)--(9,5);
\draw  [dotted] (0,6)--(9,6);
\draw  [dotted] (0,7)--(9,7);
\draw  [dotted] (0,8)--(9,8);
\draw  [dotted] (0,9)--(9,9);

\draw (9,0) node {$\bullet$};
\draw (8,1) node {$\bullet$};
\draw (7,2) node {$\bullet$};
\draw (6,3) node {$\bullet$};
\draw (5,4) node {$\bullet$};
\draw (4,5) node {$\bullet$};
\draw (3,6) node {$\bullet$};
\draw (2,7) node {$\bullet$};
\draw (1,8) node {$\bullet$};
\draw (0,9) node {$\bullet$};

\draw (7,1) node {$\bullet$};
\draw (6,2) node {$\bullet$};
\draw (5,3) node {$\bullet$};
\draw (4,4) node {$\bullet$};
\draw (3,5) node {$\bullet$};
\draw (2,6) node {$\bullet$};
\draw (1,7) node {$\bullet$};

\draw (5,2) node {$\bullet$};
\draw (4,3) node {$\bullet$};
\draw (3,4) node {$\bullet$};
\draw (2,5) node {$\bullet$};

\draw (3,3) node {$\bullet$};

\draw (9,-0.5) node {$\deg_q$};
\draw (-0.5,9) node {$\deg_t$};

\draw (0,2.5)--(9,2.5);
\draw (2.5,0)--(2.5,9);
\draw (3.5,0)--(3.5,9);

\draw (5.5,0.5) node {$A_{i,k}$};
\draw (0.5,5.5) node {$B_{i,k}$};
\draw (5.5,5.5) node {$C_{i,k}$};
\draw (3,6.5) node {$D_k$};
\end{tikzpicture}
\end{center}
\end{example}

\begin{prop}
\label{prop: bijection}
The generators of $J$ are in bijection with the integer lattice points inside the trapezoid bounded by the following inequalities:
$$2d_1+d_2\geq x+y \geq d_1+d_2,\
x+2y \geq 2d_1+d_2,\
2x+y \geq 2d_1+d_2.$$

\end{prop}
\begin{proof}
    It is easy to check that all of the generators of $J$ satisfy these inequalities on their bidegrees $(x,y)=(\deg_q,\deg_t)$. Since the generators all have unique $(q,t)$-bidegree, it is sufficient to count the number of points in the integer lattice and check that it is the same as the number of generators.

    We will count the lattice points going by diagonals, starting with the top diagonal $x+y = 2d_1+d_2$. On this diagonal, $x+2y \geq 2d_1+d_2$ and $2x+y \geq 2d_1+d_2$ are trivially satisfied, since $x,y \geq 0$. So $x$ and $y$ can both range from $0$ to $2d_1+d_2$, and there are $2d_1+d_2+1$ points on this diagonal.

    On the next diagonal, $x+y = 2d_1+d_2-1$, we have that $x+2y = 2d_1+d_2-1+y$ and $2x+y = 2d_1+d_2-1+x$ are both at least $2d_1+d_2$. This implies that $x,y \geq 1$ so we have points
    $$(1,2d_1+d_2-2),\
    (2, 2d_1+d_2-1),\ \ldots
    (2d_1+d_2-2, 1),$$
    which amounts to 3 less points than the first diagonal. If we keep going, each diagonal will have 3 less points than the last, until the final diagonal $x+y = d_1+d_2$, which will have $d_2-d_1+1$ points. If we index the diagonal $x+y = 2d_1+d_2-j$ by $j,$ $j$ will range from $0$ to $d_1$. So in total the number of points in the lattice is

    $$\sum_{j=0}^{d_1} [2d_1 +d_2 + 1 - 3j].$$

    If we count the generators of $J$ as laid out in Theorem \ref{thm: generators}, we get

    $$A_{i,j}: \sum_{j=0}^{d_1} (d_1-j),\quad \, B_{i,j}: \sum_{j=0}^{d_1} (d_1-j),\quad 
    C_{i,j}: \sum_{j=0}^{d_1} (d_2-j), \quad D_j: \sum_{j=0}^{d_1} 1,$$
 and combining the sums, we get the same count. So we have shown the desired bijection.

\end{proof}

\begin{proof}[Proof of Theorem \ref{thm: Hilbert}(b)]
    In Example 1.2 in \cite{qt}, the authors show that 
    $$F(d_1,d_2,d_3) = [2d_1+d_2+1]_{q,t} + qt[2d_1+d_2-2] + \dots + q^{d_1}t^{d_1}[d_2-d_1+1]_{q,t},$$
    where $[n+1]_{q,t} := q^n + q^{n-1}t + \dots + qt^{n-1} + t^n$.

    We can see from the proof of Proposition \ref{prop: bijection} that this exactly matches up with the coordinates of the lattice points, grouped by diagonals, and therefore by Proposition \ref{prop: bijection}, this is in bijection with the generators of $J$ and their $(q,t)$ degrees.

    Any choice of basis of $J(d_1,\ldots,d_n)/\mathfrak{m}J(d_1,\ldots,d_n)$ lifts to a set of generators of $J$ by Nakayama's lemma. So the Hilbert series of $J(d_1,\ldots,d_n)/\mathfrak{m}J(d_1,\ldots,d_n)$ is precisely the degree count of the generators of $J$, so indeed it is $F(d_1,d_2,d_3)$.
\end{proof}

\section{Proofs}
\label{sec: proofs}
\subsection{Proof of Theorem \ref{thm: generators}}
\label{proof: generators}
After doing the change of variables to $\C[a,b,c,d]$, the upshot is that we've reduced the number of variables, and we can use the fact that $J = M \cap (a-c,b-d)^{d_1}$, where $M = (a,b)^{d_1} \cap (c,d)^{d_2}$ is a monomial ideal. In this section we will find a basis for $J$ over $\C$ by characterizing when elements of $(a-c,b-d)^{d_1}$ are in the monomial ideal $M$, proving Proposition \ref{prop: basis} and by extension Theorem \ref{thm: generators}. First, we will need a few key lemmas.

Since $M$ is a monomial ideal, a polynomial $f$ is in $M$ if and only if all monomials $m$ of $f$ that  have nonzero coefficients are in $M$. If $m$ is a monomial, then let $\deg_1(m)$ be the combined $(a,b)$ degree of $m$, i.e. the sum of its $a$ and $b$ degrees, and similarly let $\deg_2(m)$ be its combined $(c,d)$ degree. Then the monomial $m$ is in $M$ if and only $\deg_1(m) \geq d_1$, and $\deg_2(m) \geq d_2$.
Note that these degrees should not be confused with $\deg_q$ and $\deg_t$ defined above.

Consider some $f \in R$ of the form
$$f = \varphi (b-d) + \psi (ad-bc) \in M$$
for some $\varphi, \psi \in R.$
Notice that for any $\gamma \in R$, we can modify the coefficients $\varphi, \psi$ by simultaneously substituting:
\begin{equation}
\label{relation}
    \varphi \rightarrow \varphi + \gamma(ad-bc), \qquad \psi \rightarrow \psi - \gamma(b-d)
\end{equation}
without changing $f$.

\begin{lemma}
\label{lem: key lemma}
If
$$f = \varphi (b-d) + \psi (ad-bc) \in M$$
for some $\varphi, \psi \in R,$
then up to the relation \eqref{relation}, we can assume that $\varphi \in M$.
\end{lemma}

\begin{proof}
Since $M$ is a monomial ideal, $\varphi (b-d) + \psi (ad-bc) \in M$ if and only if every monomial term of $\varphi (b-d) + \psi (ad-bc)$ is in $M$. If $m$ is a monomial in the expansion of this expression, either $m$ is in $M$, or $m$ cancels with some other monomial. 

When we expand, all monomials come in pairs from distributing $b-d$ or $ad-bc$. These pairs look like $m \left(\frac{b}{d}\right) - m$ or $m \left(\frac{ad}{bc}\right) - m$, where each term is appropriately divisible so that there are no denominators. If one of these monomials cancels with another monomial, that other monomial also must be part of a pair like the above. For example,

$$
\frac{m}{cd}(ad-bc) + \frac{m}{d}(b-d)
= m \left(\frac{b}{d}\right) \left(\frac{ad}{bc}\right) - m \left(\frac{b}{d}\right) + m \left(\frac{b}{d}\right) - m
= m \left(\frac{b}{d}\right) \left(\frac{ad}{bc}\right) - m.
$$

We can continue to follow a chain of cancellation until either we get two terminal monomials that do not cancel with anything, or we eventually reach a monomial that cancels with the starting monomial $m$, creating a cycle. We can visualize these chains of cancellations by oriented paths in a 2 dimensional lattice. Vertices represent monomials, vertical edges represent a difference of two monomials of the form $m \left(\frac{b}{d}\right) - m$, and horizontal edges represent a difference of the form $m \left(\frac{ad}{bc}\right) - m$. The full path represents the sum of all the pairs of monomials represented by each edge, and the end vertices of the path are the terminal monomials. For example, the above cancellation can be represented by the path:

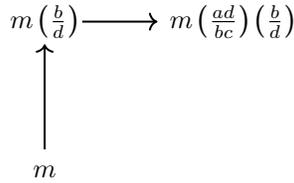
\tikzset{every picture/.style={line width=0.75pt}} 
\begin{center}
\begin{tikzpicture}
\draw (0,0) node {$m$};
\draw (0,2) node {$m\left(\frac{b}{d}\right)$};
\draw (2.5,2) node {$m\left(\frac{ad}{bc}\right)\left(\frac{b}{d}\right)$};
\draw [->] (0,0.3)--(0,1.7);
\draw [->] (0.5,2)--(1.5,2);
\end{tikzpicture}
\end{center}

Cycles in these paths correspond exactly to the equivalence 
$$
\varphi(b-d) + \psi(ad-bc)  = (\varphi + \gamma(ad-bc))(b-d) + (\psi - \gamma(b-d))(ad-bc).
$$

To see this, let $m$ be a monomial. A cycle looks like this:

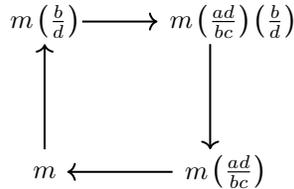
\tikzset{every picture/.style={line width=0.75pt}} 
\begin{center}
\begin{tikzpicture}
\draw (0,0) node {$m$};
\draw (0,2) node {$m\left(\frac{b}{d}\right)$};
\draw (2.5,2) node {$m\left(\frac{ad}{bc}\right)\left(\frac{b}{d}\right)$};
\draw (2.4,0) node {$m\left(\frac{ad}{bc}\right)$};
\draw [<-] (0,1.7)--(0,0.3);
\draw [<-] (1.5,2)--(0.5,2);
\draw [<-] (0.3,0)--(1.7,0);
\draw [<-] (2.2,0.3)--(2.2,1.7);
\end{tikzpicture}
\end{center}

This corresponds to the identity
$$ 
\frac{m}{d}(b-d) + \frac{m}{cd}(ad-bc) - \frac{ma}{bc}(b-d) - \frac{m}{bc}(ad-bc) = 0.
$$

For all of these terms to be monomials, $b,c,d$ must all divide $m$. Multiplying through by $bcd$ and grouping, we get

$$-m(ad-bc)(b-d) + m(b-d)(ad-bc) = 0.$$

Adding this cycle corresponds to using the above equivalence with $\gamma = m$. Adding any number of these cycles along our path corresponds to modifying the coefficients $\varphi$ and $\psi$ with relation \eqref{relation} without changing the overall sum $f$.

We can add and subtract this square loop to any path in order to both eliminate any loops, and to reorder cancellation. For example:

\tikzset{every picture/.style={line width=0.75pt}} 
\begin{center}
\begin{tikzpicture}
\draw (0,0) node {$m$};
\draw (0,2) node {$m\left(\frac{b}{d}\right)$};
\draw (2.5,2) node {$m\left(\frac{ad}{bc}\right)\left(\frac{b}{d}\right)$};
\draw [->] (0,0.3)--(0,1.7);
\draw [->] (0.5,2)--(1.5,2);

\draw (3.8,1) node {\Huge =};

\draw (5,0) node {$m$};
\draw (5,2) node {$m\left(\frac{b}{d}\right)$};
\draw (7.5,2) node {$m\left(\frac{ad}{bc}\right)\left(\frac{b}{d}\right)$};
\draw (7.4,0) node {$m\left(\frac{ad}{bc}\right)$};
\draw [<-] (5,1.7)--(5,0.3);
\draw [<-] (6.5,2)--(5.5,2);
\draw [<-] (5.3,0)--(6.7,0);
\draw [<-] (7.2,0.3)--(7.2,1.7);

\draw (8.8,1) node {\Huge +};

\draw (10,0) node {$m$};
\draw (12.4,0) node {$m\left(\frac{ad}{bc}\right)$};
\draw (12.5,2) node {$m\left(\frac{ad}{bc}\right)\left(\frac{b}{d}\right)$};
\draw [->] (10.3,0)--(11.7,0);
\draw [->] (12.2,0.3)--(12.2,1.7);

\end{tikzpicture}
\end{center}

So up to relation \eqref{relation}, we can make any path into one with vertical edges first and horizontal edges after, going from $m$ to $m \left(\frac{b}{d}\right)^k$ to $m \left(\frac{b}{d}\right)^k \left(\frac{ad}{bc}\right)^j$, with $k$ and $j$ possibly negative or 0.

\tikzset{every picture/.style={line width=0.75pt}} 
\begin{center}
\begin{tikzpicture}
\draw (0,0) node {$m$};
\draw (0,2) node {$m\left(\frac{b}{d}\right)^k$};
\draw (2.5,2) node {$m\left(\frac{ad}{bc}\right)^j\left(\frac{b}{d}\right)^k$};
\draw [dotted,->] (0,0.3)--(0,1.7); 
\draw [dotted,->] (0.5,2)--(1.5,2);
\end{tikzpicture}
\end{center}

This reduced path corresponds to $\varphi(b-d) + \psi(ad-bc)$, where 
$$\varphi = \frac{m \left(\frac{b}{d}\right)^k - m}{b-d}\ \mathrm{and}\ 
\psi = \frac{m\left(\frac{b}{d}\right)^{k}\left(\frac{ad}{bc}\right)^{j} - m\left(\frac{b}{d}\right)^{k}}{ad-bc}.$$
Notice for any monomial $m$, $m \left(\frac{ad}{bc}\right)^j \in M$ if and only if $m \in M$, as multiplying by $\frac{a}{b}$ does not change the combined $(a,b)$ degree of a monomial, and the same for $\frac{c}{d}$. 

Given a reduced path as above, we know that the terminal monomials $m$ and $m \left(\frac{b}{d}\right)^k \left(\frac{ad}{bc}\right)^j$ are in $M$. But since $m \left(\frac{b}{d}\right)^k \left(\frac{ad}{bc}\right)^j \in M$, by the above, $m \left(\frac{b}{d}\right)^k$ is in $M$. Now we want to show that $\varphi \in M$.

Since $m \left(\frac{b}{d}\right)^k \in M$, it follows that $\frac{m}{d^k} \in (c,d)^{d_2}$, since multiplying by $b^k$ does not change the $(c,d)$ degree. Since $m$ is in $(a,b)^{d_1},$ it follows that $\frac{m}{d^k} \in (a,b)^{d_1}$, since it has the same $(a,b)$ degree as $m$. So $\frac{m}{d^k} \in M$, and therefore indeed
$$\varphi = \frac{m}{d^k} \frac{b^k-d^k}{b-d} \in M.$$
For a general $\varphi(b-d) + \psi(ad-bc)$, we can break the terms into a sum of discrete cancellation chains,
$$\sum_i \left[\varphi_{1,i}(b-d) + \varphi_{2,i}(ad-bc)\right].$$ For each cancellation chain, we have shown that up to equivalence, $\varphi_{1,i} \in M$, and therefore 
$$\varphi = \sum_i \varphi_{1,i} \in M.$$

\end{proof}

\begin{lemma}
\label{b-d}
    We have that $\varphi(b-d) \in M$ if and only if $\varphi \in M$, and
    $\varphi(a-c) \in M$ if and only if $\varphi \in M$.
\end{lemma}
\begin{proof}
    This is essentially what the final argument of the above proof shows. If we expand the expression $\varphi(b-d)$ into monomials, then as in the proof of Lemma \ref{lem: key lemma}, we will get chains of cancellations with two terminal monomials that do not cancel, which looks like the path:

    \tikzset{every picture/.style={line width=0.75pt}} 
\begin{center}
\begin{tikzpicture}
\draw (0,0) node {$m$};
\draw (0,2) node {$m\left(\frac{b}{d}\right)^k$};
\draw [dotted,->] (0,0.3)--(0,1.7); 
\end{tikzpicture}
\end{center}

This chain corresponds to
$$\varphi = \frac{m \left(\frac{b}{d}\right)^k - m}{b-d},$$
and as before, $m, m \left(\frac{b}{d}\right)^k \in M$ together imply that $\varphi \in M$. Any general $\varphi(b-d)$ can be split into the sum of distinct cancellation chains, and thus $\varphi \in M$.
    
The same argument applies to $\varphi(a-c)$ since $M$ is symmetric in $(b,d), (a,c)$.
\end{proof}

Next, we characterize how we can best express $f \in (a-c,b-d)^{d_1}$ in order to see when $f \in M$.

\begin{lemma}
\label{alphaac}
    Any $f \in (a-c, b-d)$ can be written as 
    $$f = \alpha_1(a-c) + \alpha_2(b-d) + \alpha_3(ad-bc),$$
    where $\alpha_i \in R$ and $\alpha_1$ is a polynomial in $a$ and $c$ only.
\end{lemma}
\begin{proof}
    If $f \in (a-c,b-d)$, then
    $$f = \gamma_1(a-c) + \gamma_2(b-d)$$
    for some $\gamma_1,\gamma_2 \in R$. Observe that 
    $$b(a-c) = a(b-d) + (ad-bc),$$
    and
    $$d(a-c) = c(b-d) + (ad-bc).$$

    So by applying this to any term in $\gamma_1$ with a factor of $b$ or $d$, we can ensure that $\gamma_1$ only depends on $a$ and $c$.
\end{proof}

\begin{lemma}
    If $f \in (a-c,b-d)^{d_1} \cap M$, we can write $f = \sum_i f_i$, where
    $$f_i = \sum_{j=0}^{d_1-i} \alpha_{i,j} (a-c)^i(b-d)^{d_1-i-j}(ad-bc)^j$$
    and each $f_i \in M$.
\end{lemma}
\begin{proof}
    Consider $f \in (a-c, b-d)^{d_1}$. Following Lemma \ref{alphaac}, we can express $f$ as a linear combination of products of $(a-c), (b-d)$, and $(ad-bc)$,

$$f = \sum_{i,j} \alpha_{i,j} (a-c)^i (b-d)^{d_1-i-j} (ad-bc)^j.$$ Further, we can assume that every coefficient of a term with a factor of $(a-c)$ depends only on $a$ and $c$, because for any term with coefficient $\alpha$ with $i>0$, we can write
$$
\alpha(a-c)^i (b-d)^{d_1-i-j} (ad-bc)^j$$
$$= \left[\alpha(a-c) \right](a-c)^{i-1} (b-d)^{d_1-i-j} (ad-bc)^j$$
$$= \left[\alpha_1(a-c) + \alpha_2(b-d) + \alpha_3(ad-bc)\right](a-c)^{i-1} (b-d)^{d_1-i-j} (ad-bc)^j$$
where $\alpha_1$ only depends on $a$ and $c$ by Lemma \ref{alphaac}. We can continue this process by induction until all the modified coefficients $\alpha_{i, j}$ only depend on $a$ and $c$.

Now we will group terms of $f$ by their combined $(b,d)$ degree. For all $i>0$, since $\alpha_{i,j}$ depends only on $a$ and $c$, we know that the combined $b,d$ degree of every monomial of
$$
\alpha_{i,j} (a-c)^i (b-d)^{d_1-i-j} (ad-bc)^j
$$
is $k = \deg_t(m) = d_1-i$. If any monomial from this term cancels, it must cancel with another term with the same $(b,d)$ degree. So in fact, every monomial with $\deg_t(m) = k$ must come from the sum
$$
f_i = \sum_{j=0}^{k} \alpha_{i,j} (a-c)^i (b-d)^{k-j} (ad-bc)^j$$
$$ = (a-c)^i \sum_{j=0}^{k} \alpha_{i,j}(b-d)^{k-j} (ad-bc)^j
$$
with fixed $i$. In other words,  monomials can cancel within each $f_i$, but not between them. This implies that for all $i > 0$, each $f_i \in M$, since after internal cancellation, each $f_i$ is a sum of monomials in $M$.

Since $f = \sum_i f_i$ is in $M$ and $f_i \in M$ for all $i>0$,

$$f_0 = \sum_{j=0}^{d_1} \alpha_{i,j} (b-d)^{k-j} (ad-bc)^j$$

must also be in $M$. 

\end{proof}

By Corollary \ref{b-d}, we know that 

$$f_i = (a-c)^i\sum_{j=0}^{k} \alpha_{i,j}(b-d)^{k-j} (ad-bc)^j \in M$$
implies that
\begin{equation}
\label{eq: sum}
\sum_{j=0}^{k} \alpha_{i,j}(b-d)^{k-j} (ad-bc)^j \in M.
\end{equation}

Now we fix $i$ and look at a single $f_i$, which we will call $f$ to avoid unnecessary indices. Also let $k = d_1-i$ as before.

\begin{prop}
\label{prop: all terms in M}
    If 
    $$f = \sum_{j=0}^k \alpha_{j}(b-d)^{k-j} (ad-bc)^j \in M,$$
    for some $\alpha_j \in R$, then there exists $\alpha_j' \in R$ such that
    $$f = \sum_{j} \alpha_{j}'(b-d)^{k-j} (ad-bc)^j$$
    and each term $\alpha_{j}'(b-d)^{k-j} (ad-bc)^j$ of the sum is in $M$.
\end{prop}
\begin{proof}
    We can rewrite $f$ as 
    \begin{equation}
        \label{eq: redundant sum}
        f = \sum_{j=0}^{k-1} \left[ \varphi_j (b-d) + \psi_j(ad-bc) \right] (b-d)^{k-1-j}(ad-bc)^j,
    \end{equation}
    
    where initially $\varphi_j = \alpha_j$ for all $0\leq j \leq k-1$, $\psi_{k-1} = \alpha_k$, and the rest of the $\psi_j$'s are 0. Essentially we have taken the previous sum for $f$ and added some redundant terms; in particular, $\psi_j$ and $\varphi_{j+1}$ are coefficients for like terms for $0 \leq j \leq k-1$. So we have two relations we can use to modify the coefficients of \eqref{eq: redundant sum} without changing the sum: 
    \begin{eqnarray}
    \label{eq: Relation 1}
        \psi_j \rightarrow \psi_j + \gamma,& \qquad \varphi_{j+1} \rightarrow \varphi_{j+1} - \gamma\\
    \label{eq: Relation 2}
        \varphi_j \rightarrow \varphi_j + \gamma(ad-bc),& \quad \psi_j \rightarrow \psi_j - \gamma(b-d).
    \end{eqnarray}
The first comes from the redundant coefficients, and the second is the same relation \eqref{relation} used in Lemma \ref{lem: key lemma}.

    We will induct on $k = d_1-i$. If $k=0$, then $i=d_1$, and our sum \eqref{eq: sum} is just the single term $\alpha_{0}(a-c)^{d_1}$, which is in $M$ by assumption. Now assume by induction that if 
    $$
    \sum_{j=0}^{k-1} \alpha_{j} (b-d)^{k-1-j} (ad-bc)^{j}
    $$ 
    is in $M$, then we can modify the coefficients using only \eqref{eq: Relation 2} to get all terms \\
    $\alpha_{j} (b-d)^{k-1-j} (ad-bc)^{j}$ in $M$. We can apply this to \eqref{eq: redundant sum} with $\alpha_j = \varphi_j (b-d) + \psi_j(ad-bc).$ Using \eqref{eq: Relation 2} on the $\alpha_j$'s actually just amounts to using \eqref{eq: Relation 1} on the $\varphi_j$'s and $\psi_j$'s, as adding $\gamma(ad-bc)$ to $\alpha_j$ is the same as adding $\gamma$ to $\psi_j$, and subtracting $\gamma(b-d)$ from $\alpha_j$ is the same as subtracting $\gamma$ from $\varphi_j$. So the inductive hypothesis implies that up to \eqref{eq: Relation 1}, we can get 
    $$\left[\varphi_j(b-d) + \psi_j(ad-bc)\right](b-d)^{k-1-j}(ad-bc)^j$$
    in $M$ for all $j$. By Corollary \ref{b-d}, this implies that 
    \begin{equation}
    \label{eq: using key lemma}
    \left[\varphi_j(b-d) + \psi_j(ad-bc)\right](ad-bc)^j \in M. 
    \end{equation}
    Notice that multiplying a monomial (polynomial) by $(ad-bc)$ raises its combined $(a,b)$ degree and $(c,d)$ degree each by one. So (\ref{eq: using key lemma}) is in $M = (a,b)^{d_1} \cap (c,d)^{d_2}$ if and only if $\left[\varphi_j(b-d) + \psi_j(ad-bc)\right] \in N = (a,b)^{d_1-j} \cap (c,d)^{d_2-j}$. Now we apply Lemma \ref{lem: key lemma} on $N$ to get both $\varphi_j(b-d)$ and  $\psi_j(ad-bc)$ in $N$, and then when multiply by $(ad-bc)^j$, we get that both terms of (\ref{eq: using key lemma}) are in $M$.
    
    So we have shown that if 
    $$f = \sum_{j=0}^{k-1} \left[ \varphi_j (b-d) + \psi_j(ad-bc) \right] (b-d)^{k-1-j}(ad-bc)^j,$$
    up to relations \eqref{eq: Relation 1} and \eqref{eq: Relation 2}, we can get all terms of this sum to be in $M$. Now simply recombine like terms to get 
    $$f = \sum_{j} \alpha_{j}(b-d)^{d_1-i-j} (ad-bc)^j$$
    with all terms in $M$ as desired.
    
\end{proof}

Now we can lay out a basis for the ideal

$$J(d_1,d_2) = (a,b)^{d_1} \cap (c,d)^{d_2} \cap (a-c,b-d)^{d_1}$$

over $\C$.

\begin{proof}[Proof of Proposition \ref{prop: basis}]

We know by Proposition \ref{prop: all terms in M} that $J(d_1,d_2)$ is generated by polynomials of the form
$$\alpha_{i,j} (a-c)^i (b-d)^{d_1-i-j} (ad-bc)^j,$$ 
where $\alpha_{i,j}$ only depends on $a,c$ for $i > 0$.
Here $0\le j\le d_1$ and $0\le i\le d_1-j$.
So as a vector space, $J(d_1,d_2)$ is generated by 
$$a^{\alpha}b^{\beta}c^{\gamma}d^{\delta}(a-c)^i (b-d)^{d_1-i-j} (ad-bc)^j$$
with the conditions that $\alpha + \beta + j \geq d_1$, $\gamma + \delta + j \geq d_2$, and $\beta=\delta=0$ if $i>0$. Among these generators, the only kind of relations remaining are those that come from the fact that $bc = ad - (ad-bc)$.

If $i>0$, then this relation is irrelevant, and we get linearly independent generators of the form
$$m(a,c)A_{i,j}.$$
If $i=0$ and $\beta,\gamma>0$, then we can write
$$a^{\alpha}b^{\beta}c^{\gamma}d^{\delta} = a^{\alpha+1}b^{\beta-1}c^{\gamma-1}d^{\delta+1} + a^{\alpha}b^{\beta-1}c^{\gamma-1}d^{\delta}(ad-bc).$$
Continue reducing $bc$ this way until we end up in one of the following situations:
\begin{enumerate}
    \item $\gamma = 0$, $\beta \neq 0$, in which case we are left with 
    $$m(a,b,d)B_{\beta,j}$$
    with $0 \leq j \leq d_1$ and $1 \leq \beta \leq d_1-j$.
    \item $\beta = 0$, $\gamma  \neq 0$, in which case we are left with 
    $$m(a,c,d)C_{\gamma,j}$$
    with $0 \leq j \leq d_1$ and $1 \leq \gamma \leq d_2-j$.
    \item $\beta = 0$, $\gamma = 0$, in which case we are left with 
    $$m(a,d)D_j$$
    with $0 \leq j \leq d_1$.
    \item The exponent of $(ad-bc)$ is greater than or equal to $d_1$, in which case we are left with a linear combination of terms of the form
    $$m(a,b,c,d)(ad-bc)^{d_1}.$$
\end{enumerate}

\end{proof}

Theorem \ref{thm: generators} immediately follows from Proposition \ref{prop: basis}.

\subsection{Hilbert Series Calculation}
\label{sec: Hilb series}

Let us compute the Hilbert series using the basis in Proposition \ref{prop: basis}. 

\begin{lemma}
\label{lem: s}
We have 
$$
\sum_{\alpha+\beta\ge s}q^{\alpha}t^{\beta}=\frac{q^s}{(1-q)(1-t/q)}+\frac{t^s}{(1-t)(1-q/t)}.
$$
\end{lemma}

\begin{proof}
We have 
$$
\sum_{\alpha+\beta\ge s}q^{\alpha}t^{\beta}=\sum_{\beta=0}^{s-1}\frac{q^{s-\beta}t^{\beta}}{(1-q)}+\sum_{\beta=s}^{\infty}\frac{t^\beta}{1-q}=
$$
$$
\frac{q^s-t^s}{(1-q)(1-t/q)}+\frac{t^s}{(1-q)(1-t)}.
$$
Now we can use the identity
$$
\frac{1}{(1-q)(1-t)}-\frac{1}{(1-q)(1-t/q)}=\frac{1}{(1-t)(1-q/t)}.
$$
\end{proof}

\begin{thm}
The bigraded Hilbert series of $J(d_1,d_2)$ is equal to
$$
\frac{q^{2d_1+d_2}}{(1-q)^2(1-t/q)(1-t/q^2)}+\frac{t^{2d_1+d_2}}{(1-t)^2(1-q/t)(1-q/t^2)}+
$$
$$
\frac{q^{d_1}t^{d_2}(1+t)}{(1-q)(1-t)(1-q/t)(1-t^2/q)}+\frac{q^{d_2}t^{d_1}(1+q)}{(1-t)(1-q)(1-t/q)(1-q^2/t)}.
$$
\end{thm}

\begin{proof}
We compute the contribution of various basis elements.

\noindent {\bf 1.} The contribution of $m(a,c)A_{i,j},\ j\le d_1-1$ equals
$$
\frac{1}{(1-q)^2}\sum_{j=0}^{d_1-1}\sum_{i=1}^{d_1-j}q^{d_1+d_2-j+i}t^{d_1-i}=
$$
$$
\frac{1}{(1-q)^2}\sum_{j=0}^{d_1-1}\frac{q^{d_1+d_2-j+1}t^{d_1-1}-q^{2d_1+d_2-2j+1}t^{j-1}}{(1-q/t)}=
$$
$$
\frac{q^{d_1+d_2+1}t^{d_1-1}-q^{d_2+1}t^{d_1-1}}{(1-q)^2(1-q^{-1})(1-q/t)}-\frac{q^{2d_1+d_2+1}t^{-1}-q^{d_2+1}t^{d_1-1}}{(1-q)^2(1-q/t)(1-t/q^2)}.
$$

\noindent {\bf 2.} The set $m(a,b,d)B_{i,j}\cup m(a,d)D_{j},\ j\le d_1-1$ consists of elements
$$
a^{\alpha}b^{\beta}d^{\gamma}(b-d)^{d_1-j}(ad-bc)^{j},\ \alpha+\beta\ge d_1-j,\ \gamma\ge d_2-j,
$$
so by Lemma \ref{lem: s} we get the Hilbert series
$$
\sum_{j=0}^{d_1-1}\left[\frac{q^{d_1-j}}{(1-q)(1-t/q)}+\frac{t^{d_1-j}}{(1-t)(1-q/t)}\right]\frac{q^{j}t^{d_1+d_2-j}}{(1-t)}=
$$
$$
\sum_{j=0}^{d_1-1}\left[\frac{q^{d_1}t^{d_1+d_2-j}}{(1-t)(1-q)(1-t/q)}+\frac{q^{j}t^{2d_1+d_2-2j}}{(1-t)^2(1-q/t)}\right]=
$$
$$
\frac{q^{d_1}t^{d_1+d_2}-q^{d_1}t^{d_2}}{(1-t^{-1})(1-t)(1-q)(1-t/q)}+\frac{t^{2d_1+d_2}-q^{d_1}t^{d_2}}{(1-q/t^2)(1-t)^2(1-q/t)}
$$
\noindent{\bf 3.} Similarly, for $m(a,c,d)C_{i,j}\cup m(a,d)D_{j},\ j\le d_1-1$ we get
$$
a^{\alpha}c^{\beta}d^{\gamma}(b-d)^{d_1-j}(ad-bc)^{j},\ \alpha\ge d_1-j,\ \beta+\gamma\ge d_2-j,
$$
so the Hilbert series equals
$$
\sum_{j=0}^{d_1-1}\left[\frac{q^{d_2-j}}{(1-q)(1-t/q)}+\frac{t^{d_2-j}}{(1-t)(1-q/t)}\right]\frac{q^{d_1}t^{d_1}}{(1-q)}=
$$
$$
\frac{q^{d_1+d_2}t^{d_1}-q^{d_2}t^{d_1}}{(1-q^{-1})(1-q)^2(1-t/q)}+\frac{q^{d_1}t^{d_1+d_2}-q^{d_1}t^{d_2}}{(1-t^{-1})(1-q)(1-t)(1-q/t)}.
$$
\noindent {\bf 4.} We overcount by $m(a,d)D_{j},\ j\le d_1-1$ which contribute
$$
\frac{1}{(1-q)(1-t)}\sum_{j=0}^{d_1-1}q^{d_1}t^{d_1+d_2-j}=\frac{q^{d_1}t^{d_1+d_2}-q^{d_1}t^{d_2}}{(1-t^{-1})(1-q)(1-t)}
$$
{\bf 5.} For $j=d_1$ we have special terms
$$
a^{\alpha}b^{\beta}c^{\gamma}d^{\delta}(ad-bc)^{d_1},\ \gamma+\delta\ge d_2-d_1,
$$
which contribute
$$
\frac{1}{(1-q)(1-t)}\left[\frac{q^{d_2-d_1}}{(1-q)(1-t/q)}+\frac{t^{d_2-d_1}}{(1-t)(1-q/t)}\right] q^{d_1}t^{d_1}.
$$
{\bf 6.} Finally, we collect the coefficients at similar terms: 
$$
q^{d_1+d_2}t^{d_2}\left[\frac{qt^{-1}}{(1-q)^2(1-q^{-1})(1-q/t)}+\frac{1}{(1-q^{-1})(1-q)^2(1-t/q)}\right]=0;
$$
$$
q^{d_2}t^{d_1}[-\frac{1}{(1-q)^2(1-q^{-1})(1-q/t)}+\frac{1}{(1-q)^2(1-q/t)(1-t/q^2)}-
$$
$$
\frac{1}{(1-q^{-1})(1-q)^2(1-t/q)}+\frac{1}{(1-q)(1-t)(1-q)(1-t/q)}]=\frac{q^{d_2}t^{d_1}(1+q)}{(1-t)(1-q)(1-t/q)(1-q^2/t)}.
$$
$$
q^{2d_1+d_2}\left[-\frac{qt^{-1}}{(1-q)^2(1-q/t)(1-t/q^2)}\right]+\frac{t^{2d_1+d_2}}{(1-q/t^2)(1-t)^2(1-q/t)}=
$$
$$
\frac{q^{2d_1+d_2}}{(1-q)^2(1-t/q)(1-t/q^2)}+\frac{t^{2d_1+d_2}}{(1-t)^2(1-q/t)(1-q/t^2)};
$$
$$
q^{d_1}t^{d_1+d_2}[\frac{1}{(1-t^{-1})(1-t)(1-q)(1-t/q)}+\frac{1}{(1-t^{-1})(1-q)(1-t)(1-q/t)}-
$$
$$\frac{1}{(1-t^{-1})(1-q)(1-t)}]=0;
$$
$$
q^{d_1}t^{d_2}[-\frac{t^{-1}}{(1-t^{-1})(1-t)(1-q)(1-t/q)}-\frac{qt^{-2}}{(1-q/t^2)(1-t)^2(1-q/t)}-$$ 
$$
\frac{t^{-1}}{(1-t^{-1})(1-q)(1-t)(1-q/t)}+\frac{t^{-1}}{(1-t^{-1})(1-q)(1-t)}+\frac{1}{(1-q)(1-t)(1-t)(1-q/t)}
]=
$$
$$
\frac{q^{d_1}t^{d_2}(1+t)}{(1-q)(1-t)(1-q/t)(1-t^2/q)}.
$$
\end{proof}

\section{Braid Recursion}
\label{sec: recursion}

In this section we compute Khovanov-Rozansky homology for the link $L_{\gamma}$ associated to $\gamma$ as in \eqref{gamma def} for $n=3$. It is the closure of the braid $\beta_{d_1,d_2} = (\FT_2)^{d_2-d_1} \cdot (\FT_3)^{d_1}$ where $\FT_2=\sigma_1^{2}$ and $\FT_3=(\sigma_1^2)\sigma_2\sigma_1^2\sigma_2$ are the full twist braids on two and three strands respectively.
We will also use the Jucys-Murphy braids $\JM_2=\FT_2$ and $\JM_3=\FT_2^{-1}\FT_3=\sigma_2\sigma_1^2\sigma_2$, so that $\beta_{d_1,d_2}=\JM_2^{d_2}\JM_3^{d_1}$.
It is well-known that $\JM_2$ and $\JM_3$ commute in the braid group.

We use the recursion for triply graded Khovanov-Rozansky homology defined by Elias and Hogancamp in \cite{torus links}, as described in \cite{Link homology}. Each braid-like diagram represents a complex of Soergel bimodules, where crossings correspond to Roquier complexes, and stacking a complex above (or below) represents taking a tensor product with that complex on the right (or left respectively). The $K_n$ are certain complexes of Soergel bimodules defined in \cite{torus links} that satisfy the relations below.

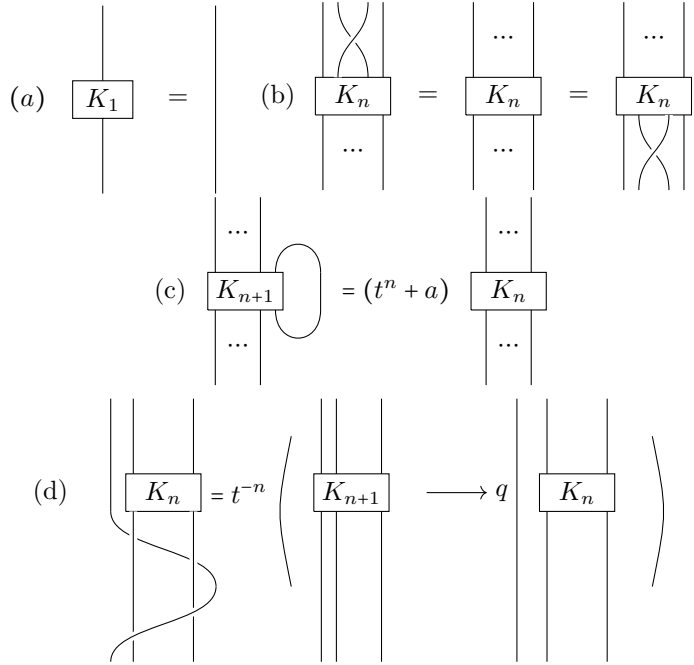
\begin{center}
\begin{tikzpicture}
\draw (-1,1.25) node {$(a)$};
\draw (0,0)--(0,1)--(-0.4,1)--(-0.4,1.5)--(0.4,1.5)--(0.4,1)--(0,1);
\draw (0,1.5)--(0,2.5);
\draw (0,1.25) node {$K_1$};
\draw (1,1.25) node {=};
\draw (1.5,0)--(1.5,2.5);
\end{tikzpicture}
\quad
\begin{tikzpicture}
\draw (-1,1.25) node {(b)};
\draw  (-0.5,1)--(-0.5,1.5)--(0.5,1.5)--(0.5,1)--(-0.5,1);
\draw (-0.4,0)--(-0.4,1);
\draw (0,0.5) node {$\cdots$};
\draw (0.4,0)--(0.4,1);
\draw (0,1.25) node {$K_n$};
\draw (-0.4,1.5)--(-0.4,2.5);
\draw (0.4,1.5)--(0.4,2.5);
\draw (-0.2,2.5)..controls (-0.2,2) and (0.2,2)..(0.2,1.5);
\draw[white,line width=3] (-0.2,1.5)..controls (-0.2,2) and (0.2,2)..(0.2,2.5);
\draw (-0.2,1.5)..controls (-0.2,2) and (0.2,2)..(0.2,2.5);
\draw (1,1.25) node {=};
\draw  (1.5,1)--(1.5,1.5)--(2.5,1.5)--(2.5,1)--(1.5,1);
\draw (1.6,0)--(1.6,1);
\draw (2,0.5) node {$\cdots$};
\draw (2.4,0)--(2.4,1);
\draw (2,1.25) node {$K_n$};
\draw (1.6,1.5)--(1.6,2.5);
\draw (2.4,1.5)--(2.4,2.5);
\draw (2,2) node {$\cdots$};
\draw (3,1.25) node {=};
\draw  (3.5,1)--(3.5,1.5)--(4.5,1.5)--(4.5,1)--(3.5,1);
\draw (3.6,0)--(3.6,1);
\draw (4,2) node {$\cdots$};
\draw (4.4,0)--(4.4,1);
\draw (4,1.25) node {$K_n$};
\draw (3.6,1.5)--(3.6,2.5);
\draw (4.4,1.5)--(4.4,2.5);
\draw (3.8,1)..controls (3.8,0.5) and (4.2,0.5)..(4.2,0);
\draw[white,line width=3] (3.8,0)..controls (3.8,0.5) and (4.2,0.5)..(4.2,1);
\draw (3.8,0)..controls (3.8,0.5) and (4.2,0.5)..(4.2,1);
\end{tikzpicture}
\quad
\begin{tikzpicture}
\draw (-1,1.25) node {(c)};
\draw  (-0.5,1)--(-0.5,1.5)--(0.5,1.5)--(0.5,1)--(-0.5,1);
\draw (-0.4,0)--(-0.4,1);
\draw (-0.1,0.5) node {$\cdots$};
\draw (-0.1,2) node {$\cdots$};
\draw (0,1.25) node {$K_{n+1}$};
\draw (-0.4,1.5)--(-0.4,2.5);
\draw (0.4,1.5)..controls (0.4,2) and (1,2)..(1,1.5);
\draw (0.4,1)..controls (0.4,0.5) and (1,0.5)..(1,1);
\draw (0.2,0)--(0.2,1);
\draw (0.2,1.5)--(0.2,2.5);
\draw (1,1)--(1,1.5);
\draw (2,1.25) node {$=(t^n+a)$};
\draw  (3,1)--(3,1.5)--(4,1.5)--(4,1)--(3,1);
\draw (3.2,0)--(3.2,1);
\draw (3.8,0)--(3.8,1);
\draw (3.5,0.5) node {$\cdots$};
\draw (3.5,2) node {$\cdots$};
\draw (3.5,1.25) node {$K_{n}$};
\draw (3.2,1.5)--(3.2,2.5);
\draw (3.8,1.5)--(3.8,2.5);
\end{tikzpicture}
\end{center}

\begin{center}
\begin{tikzpicture}
\draw (-1.5,1.25) node {(d)};
\draw  (-0.5,1)--(-0.5,1.5)--(0.5,1.5)--(0.5,1)--(-0.5,1);
\draw (0,1.25) node {$K_{n}$};
\draw  (-0.7,2.5)--(-0.7,1);
\draw (-0.4,1.5)--(-0.4,2.5);
\draw (0.4,1.5)--(0.4,2.5);
\draw  (-0.7,1)..controls (-0.7,0.5) and (0.7,0.5)..(0.7,0);
\draw[white,line width=3] (-0.4,1)--(-0.4,0);
\draw [white,line width=3]  (0.4,1)--(0.4,0);
\draw  (-0.4,1)--(-0.4,-1);
\draw    (0.4,1)--(0.4,-1);
\draw [white,line width=3]    (0.7,0)..controls (0.7,-0.5) and (-0.7,-0.5)..(-0.7,-1);
\draw  (0.7,0)..controls (0.7,-0.5) and (-0.7,-0.5)..(-0.7,-1);

\draw (1,1.25) node {$=t^{-n}$};
\draw (1.7,2)..controls (1.5,1) and (1.5,1)..(1.7,0);
\draw  (2,1)--(2,1.5)--(3,1.5)--(3,1)--(2,1);
\draw (2.5,1.25) node {$K_{n+1}$};
\draw (2.1,1.5)--(2.1,2.5);
\draw (2.3,1.5)--(2.3,2.5);
\draw (2.9,1.5)--(2.9,2.5);
\draw (2.1,1)--(2.1,-1);
\draw (2.3,1)--(2.3,-1);
\draw (2.9,1)--(2.9,-1);
\draw [->] (3.5,1.25)--(4.3,1.25);
\draw (4.5,1.25) node {$q$};
\draw  (5,1)--(5,1.5)--(6,1.5)--(6,1)--(5,1);
\draw (5.5,1.25) node {$K_{n}$};
\draw (5.1,1.5)--(5.1,2.5);
\draw (5.9,1.5)--(5.9,2.5);
\draw (5.1,1)--(5.1,-1);
\draw (5.9,1)--(5.9,-1);
\draw (4.7,2.5)--(4.7,-1);
\draw (6.5,2)..controls (6.7,1) and (6.7,1)..(6.5,0);
\end{tikzpicture}
\end{center}

For this section, we think of Khovanov-Rozansky homology $\HHH$ as a certain functor from complexes of Soergel bimodules to triply graded vector spaces over $\C$. For a braid, $\HHH(\beta)$ is $\HHH$ of the complex associated to its braid diagram. For some braids, we can use the above recursions to compute $\HHH$ by starting with its braid diagram and applying these rules until all strands are closed up. We refer to \cite{Link homology} and references therein for all details.

The three gradings of $\HHH(\beta)$ are typically denoted $Q,T,$ and $A$. These are related to the lowercase $q,t,$ and $a$ by the change of variables
$$q = Q^2, \quad t = T^2Q^{-2}, \quad a=AQ^{-2}.$$

\begin{defn}
    For any braid or braid-like diagram $\beta$, we say that $\HHH(\beta)$ is parity if it is supported in only even $T$ degrees. We will also say $\beta$ itself is parity if $\HHH(\beta)$ is parity.
\end{defn}

\begin{fact}[\cite{Link homology}]
    If both complexes on the right-hand side of step (d) are parity, then the left-hand side is also parity, and we can replace the map with addition (corresponding to a direct sum of modules). 
\end{fact}

Essentially, we do our calculations assuming that the right-hand side of (4) is parity. If we can successfully reduce $\beta$ down to complexes that are known to be parity, then we have shown that $\beta$ and all of the intermediate complexes were also parity, and we can compute $\HHH(\beta)$ along the way. We will often take the $a$-degree 0 component of $\HHH$, denoted $\HHH^{a=0}$.

\begin{thm}
\label{thm: parity}
For all $d_2\ge d_1$  $\HHH^{a=0}(\beta_{d_1,d_2})$ is parity.
\end{thm}

\begin{proof}
To make the calculation simpler, we translate the braid pictures into equations, so for example, (d) translates to
$$K_n \cdot \JM_{n+1} = t^{-n}K_{n+1} + qt^{-n} K_n,$$
where $\JM_i$ is the Jucys-Murphy braid on $i$ strands. In general when we write $K_i, \JM_i,$ or $\FT_i$ (for a full twist), it is always on the first $i$ strands, and multiplication represents vertical stacking from bottom to top (or equivalently a tensor product of Soergel bimodules).

We can resolve the first $\FT_2$ by the process:
$$\FT_2 = K_1 \cdot \FT_2 = K_1 \cdot \JM_2 = t^{-1}K_2 + qt^{-1}K_1.$$
Since $K_2$ absorbs any crossings on the first two strands, this leaves us with 
$$ 
\beta_{d_1,d_2} = t^{-1}K_2 \cdot (\FT_3)^{d_1} + qt^{-1}\beta_{d_1, d_2-1}.
$$ 
Now we resolve $\alpha_{d_1} = K_2 \cdot (\FT_3)^{d_1}$ by first writing $\FT_3$ as the product $\JM_2 \cdot \JM_3$. The $\JM_2$ gets absorbed by $K_2$, so we get
$$K_2 \cdot (\FT_3)^{d_1} = K_2 \cdot \JM_3 \cdot (\FT_3)^{d_1-1} = t^{-2}K_3 + qt^{-2} K_2 (\FT_3)^{d_1-1},$$
where the $K_3$ has absorbed the rest of the $\FT_3$'s. When we close up $K_3$, we get $t^3$, and $\alpha_{0} = K_2$ closes up to $\frac{t}{1-q}$. So we get a simple recursion:
$$
\HHH^{a=0}(\alpha_{d_1}) = t + qt^{-2} \HHH^{a=0}(\alpha_{d_1-1}),
$$ 
with $\HHH^{a=0}(\alpha_0) = \frac{t}{1-q}$. We can write this in a closed form  as 
$$\HHH^{a=0}(\alpha_{d_1}) = t \left[\frac{1 - (qt^{-2})^{d_1}}{1-qt^{-2}} + \frac{(qt^{-2})^{d_1}}{q-1} \right].$$

So overall we have the recursive relation  
\begin{equation}
\label{eq: recursion}
\HHH^{a=0}(\beta_{d_1,d_2}) = \left[\frac{1 - (qt^{-2})^{d_1}}{1-qt^{-2}} + \frac{(qt^{-2})^{d_1}}{q-1} \right] + qt^{-1}\HHH^{a=0}(\beta_{d_1, d_2-1}),
\end{equation}
with $\beta_{d_1,0} = (FT_3)^{d_1} = T(3d_1, 3)$, which is parity (with known homology) by \cite{torus links}. 
\end{proof}

To compare this to $H_*(\Sp_{\gamma})$, it can be checked by direct computation that the Hilbert series $H(d_1, d_2)$ of $\Sp_{\gamma}$ satisfies essentially the same recursion \eqref{eq: recursion}. But we can also apply a theorem of Gorsky and Hogancamp (Proposition 5.5 in \cite{GH}). Here $\HY$ is the y-ified Khovanov Rozansky homology defined in \cite{GH}.

\begin{thm}[\cite{GH}]
\label{thm: HY(beta)}
    Assume that $\beta = JM_1^{d_n} \ldots JM_n^{d_1},\ d_n\ge d_{n-1}\ge \cdots \ge d_1$, and $\HHH^{a=0}(\beta)$ is parity. Then:

    \begin{enumerate}
        \item $\HY^{a=0}(\beta) = \HHH^{a=0}(\beta) \otimes \C[y_1, \dots, y_n]$ and $\HHH^{a=0}(\beta) = \HY^{a=0}(\beta)/(y)$
        \item $I(d_1, \ldots, d_n) \subseteq \HY^{a=0}(\beta) \subseteq J'(d_1,\ldots,d_n)$, where $I$ is the product
        $$I(d_1, \ldots, d_n) = J'(1, \ldots, 1)^{d_1} \cdot J'(0, 1, \ldots, 1)^{d_2-d_1} \cdot \ldots \cdot J'(0, \ldots, 0, 1)^{d_{n-1} - d_{n-2}}.$$
    \end{enumerate}
\end{thm}

In our case $\beta$ can be expressed exactly as above, and Theorem \ref{thm: parity} along with Corollary \ref{cor: J111} implies that for $n=3$:
$$\HY^{a=0}(L_{\gamma}) = J'(d_1,d_2).$$ Note the analogy between the relationship of $\HY$ to $\HHH$ in statement (1) of Theorem \ref{thm: HY(beta)}, and the relationship of $H_*^T(\Sp_{\gamma})$ to $H_*(\Sp_{\gamma})$ due to Fact \ref{fact: mod y}. The result is the following weaker version of Conjecture \ref{conj: ORS}.

\begin{thm}
    \label{thm: ORS}
    For $n=3$ and $\gamma$ as in $\ref{gamma def}$,
    $$\HY^{a=0}(L_{\gamma})\otimes_{\C[\mathbf{x}]}\C[\mathbf{x},\mathbf{x}^{\pm}]=\Delta H_*^T(\Sp_{\gamma})$$
and $$\HHH^{a=0}(L_{\gamma})\otimes_{\C[\mathbf{x}]}\C[\mathbf{x},\mathbf{x}^{\pm}] = H_*(\Sp_{\gamma}).$$
\end{thm}

In order to show Theorem \ref{thm: ORS}, we used Corollary \ref{cor: J111}, but one could instead show that the link splitting map defined in \cite{GH} is canonical for $\beta$ as above, analogous to Proposition 6.11 in \cite{GH}. This means that Theorem \ref{thm: ORS} could be generalized to higher $n$ solely by showing that $\beta$ is parity, such as with more explicit recursions. We are optimistic that this can be done for $n=4$ and possibly for higher $n$.
\begin{appendix}

\section{Generalized \texorpdfstring{$(q,t)$}{(q,t)}-Catalan Numbers and the Fundamental Domain for \texorpdfstring{$n=3$}{n=3}\newline by Eugene Gorsky and Joshua P. Turner}

\maketitle

In \cite{Chen} Zongbin Chen introduced a notion of the fundamental domain for an unramified affine Springer fiber, which captures the behavior of cells in $\Sp_{\gamma}$ under translations by the lattice $\Lambda$.  More precisely,  for $n\le 4$ (and conjecturally in general)  $\Sp_{\gamma}$  admits a cell decomposition with cells parametrized by the lattice $\Lambda$. There is one torus fixed point in each cell. In general, the dimension of a cell is a complicated piecewise-linear function on $\Lambda$ which stabilizes outside of the fundamental domain $\CP$. The cells corresponding to points in $\Lambda$ outside $\CP$ can be obtained by translation of cells  corresponding to points in $\CP$.

At the same time, by Theorem \ref{thm: intro J} the (non-equivariant) homology of $\Sp_{\gamma}$ as a module over $\Lambda$ is captured by $\cJ/(y)\cJ$ as a module over $\mathbb{C}[x_1^{\pm},\ldots,x_n^{\pm}]$, so the cells in $\CP$ should correspond to the generators of $\cJ$, and (following Conjecture \ref{intro conj: Hilbert}) to the generalized $(q,t)$-Catalan numbers $F(d_n,\ldots,d_1)$. 

In this appendix we explore the definition and some general properties of  $\CP$ and establish its precise relation with the generalized $(q,t)$-Catalan numbers for $n=3$. We hope to generalize this to higher $n$ in future work. 

\subsection{The Fundamental Domain}

We define the action of $S_n$ on $\mathbb{R}^n$ by $\sigma(x_1,\ldots,x_n)=(x_{\sigma(1)},\ldots,x_{\sigma(n)})$. Note that for the basis vectors $\ee_i$ we have $\sigma e_i=e_{\sigma^{-1}(i)}.$

We start with the matrix $\gamma=\mathrm{diag}(\gamma_1,\ldots,\gamma_n)$ as in \eqref{gamma def}, where $\gamma_i$ are pairwise distinct monomials with order $d_i$. We will always assume that $d_1\le \ldots\le d_n$. Define $d_{ij}=\min(d_i,d_j)$ for $i\neq j$, noting that $d_{ij}$ is the order of $\gamma_i-\gamma_j$.

\begin{defn}
(Compare with \cite[Proposition 2.8]{Chen})
We define the polytope $\CP(d_1,\ldots,d_n)$  as the convex hull of the points $p_{\sigma}=\sigma(b_{1,\sigma},\ldots,b_{n,\sigma})$ where:
$$
b_{i,\sigma}=\sum_{j<i}d_{\sigma^{-1}(i),\sigma^{-1}(j)},\ \sigma\in S_n.
$$
\end{defn}

\begin{example}
For $n=2$ we get two points $p_{e}=(0,d_{1})$ and $p_{(1\ 2)}=(d_{1},0)$, and $\CP$ is the segment connecting them.
\end{example}

\begin{example}
For $n=3$ we get 6 points shown in the following table
\begin{center}
\begin{tabular}{c|c|c}
$\sigma$ & $(b_{1,\sigma},b_{2,\sigma},b_{3,\sigma})$ & $p_{\sigma}$\\
\hline
$e$ & $(0,d_1,d_1+d_2)$ & $(0,d_1,d_1+d_2)$\\
\hline
$(1\ 2)$ & $(0, d_1,  d_1+d_2)$ & $(d_1,0,d_1+d_2)$\\
\hline
$(1\ 3)$ & $(0, d_2, 2d_1)$ & $(2d_1, d_2, 0)$\\
\hline
$(2\ 3)$ & $(0, d_1, d_1+d_2)$ & $(0, d_1+d_2, d_1)$ \\
\hline
$(1\ 2\ 3)$ & $(0, d_1, d_1+d_2)$ & $(d_1, d_1+d_2, 0)$ \\
\hline
$(1\ 3\ 2)$ & $(0, d_2, 2d_1)$ & $(2d_1, 0, d_2)$\\
\hline
\end{tabular}
\end{center}
\begin{figure}[ht!]
\begin{tikzpicture}
\filldraw[color=lightgray] (3,0)--(4.5,2.598)--(2,6.928)--(-1,6.928)--(3,0);
\draw (0,0)--(3,0)--(4.5,2.598)--(2,6.928)--(-1,6.928)--(-2.5,4.33)--(0,0);
\draw (0,0) node {$\bullet$};
\draw (1,0) node {$\bullet$};
\draw (2,0) node {$\bullet$};
\draw (3,0) node {$\bullet$};

\draw (-0.5,0.866) node {$\bullet$};
\draw (0.5,0.866) node {$\bullet$};
\draw (1.5,0.866) node {$\bullet$};
\draw (2.5,0.866) node {$\bullet$};
\draw (3.5,0.866) node {$\bullet$};

\draw (-1,1.732) node {$\bullet$};
\draw (0,1.732) node {$\bullet$};
\draw (1,1.732) node {$\bullet$};
\draw (2,1.732) node {$\bullet$};
\draw (3,1.732) node {$\bullet$};
\draw (4,1.732) node {$\bullet$};

\draw (-1.5,2.598) node {$\bullet$};
\draw (-0.5,2.598) node {$\bullet$};
\draw (0.5,2.598) node {$\bullet$};
\draw (1.5,2.598) node {$\bullet$};
\draw (2.5,2.598) node {$\bullet$};
\draw (3.5,2.598) node {$\bullet$};
\draw (4.5,2.598) node {$\bullet$};

\draw (-2,3.464) node {$\bullet$};
\draw (-1,3.464) node {$\bullet$};
\draw (0,3.464) node {$\bullet$};
\draw (1,3.464) node {$\bullet$};
\draw (2,3.464) node {$\bullet$};
\draw (3,3.464) node {$\bullet$};
\draw (4,3.464) node {$\bullet$};

\draw (-2.5,4.33) node {$\bullet$};
\draw (-1.5,4.33) node {$\bullet$};
\draw (-0.5,4.33) node {$\bullet$};
\draw (0.5,4.33) node {$\bullet$};
\draw (1.5,4.33) node {$\bullet$};
\draw (2.5,4.33) node {$\bullet$};
\draw (3.5,4.33) node {$\bullet$};

\draw (-2,5.196) node {$\bullet$};
\draw (-1,5.196) node {$\bullet$};
\draw (0,5.196) node {$\bullet$};
\draw (1,5.196) node {$\bullet$};
\draw (2,5.196) node {$\bullet$};
\draw (3,5.196) node {$\bullet$};

\draw (-1.5,6.062) node {$\bullet$};
\draw (-0.5,6.062) node {$\bullet$};
\draw (0.5,6.062) node {$\bullet$};
\draw (1.5,6.062) node {$\bullet$};
\draw (2.5,6.062) node {$\bullet$};

\draw (-1,6.928) node {$\bullet$};
\draw (0,6.928) node {$\bullet$};
\draw (1,6.928) node {$\bullet$};
\draw (2,6.928) node {$\bullet$};
 
\draw (0,-0.3) node {\scriptsize $(0,d_1,d_1+d_2)$};
\draw (3,-0.3) node {\scriptsize $(d_1,0,d_1+d_2)$};
\draw (-3.6,4.33) node {\scriptsize $(0,d_1+d_2,d_1)$};
\draw (-2.1,6.928) node {\scriptsize $(d_1,d_1+d_2,0)$};
\draw (3.1,6.928) node {\scriptsize $(2d_1,d_2,0)$};
\draw (5.6,2.598) node {\scriptsize $(2d_1,0,d_2)$};
\end{tikzpicture}

\caption{Fundamental domain for $(d_1,d_2)=(3,5)$}
\label{fig: domain}
\end{figure}
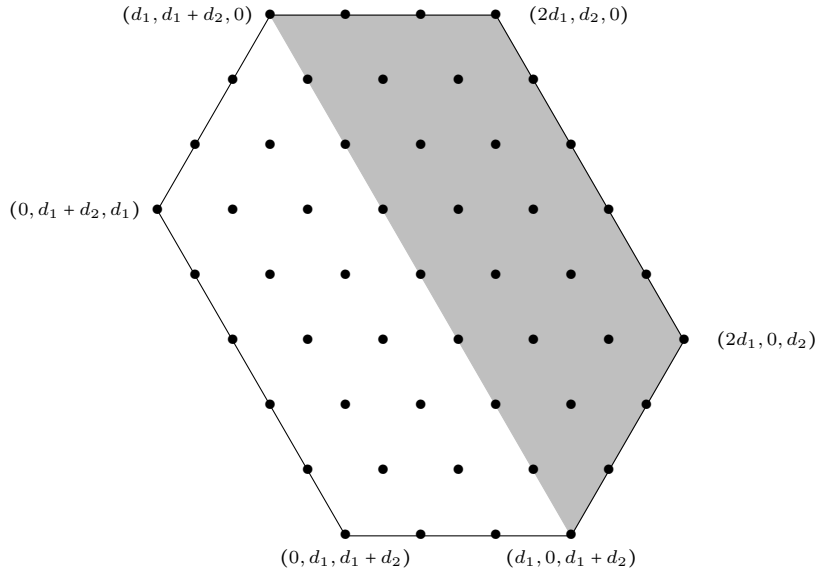
\end{example}

\begin{example}
Suppose that $d_i=d$ for all $i$. Then $b_{i,\sigma}=d(i-1)$,
$$
p_{\sigma}=(d(\sigma(1)-1),\ldots,d(\sigma(n)-1)),
$$ and $\CP$ is the standard $(n-1)$-dimensional permutahedron dilated $d$ times.
\end{example}

We now establish some general properties of $\CP$.

\begin{lemma}
The polytope $\CP(d_1,\ldots,d_n)$ is contained in the hyperplane $\sum x_i=\sum_{i<j}d_{ij}$.
\end{lemma}

\begin{proof}
For any $\sigma$ we have $\sum_{i} b_{i,\sigma}=\sum_{j<i} d_{\sigma^{-1}(i),\sigma^{-1}(j)}=\sum_{i<j}d_{ij}.$
\end{proof}

\begin{prop}
Let $\ee_i$ denote the $i$-th basis vector. Then $\CP(d_1,\ldots,d_n)$ is the Minkowski sum of $\binom{n}{2}$ segments connecting $d_{ij}\ee_i$ and $d_{ij}\ee_j$.
\end{prop}

\begin{proof}
Let $\CP'(d_1,\ldots,d_n)$ denote the above Minkowski sum.
We can write
$$
p_{\sigma}=\sigma\left(\sum_{i<j}d_{\sigma^{-1}(i),\sigma^{-1}(j)}\ee_i\right)=
\sum_{i<j}d_{\sigma^{-1}(i),\sigma^{-1}(j)}\ee_{\sigma^{-1}(i)}=\sum_{\sigma(i)<\sigma(j)}d_{ij}\ee_{i}.
$$
Given a permutation $\sigma$ and $i<j$, we can choose one end of the segment connecting $d_{ij}\ee_i$ and $d_{ij}\ee_j$ as follows: if $\sigma(i)<\sigma(j)$, we choose $\ee_i$, otherwise we choose $\ee_j$. Clearly, the sum of these points equals $p_{\sigma}$, so $p_{\sigma}\in \CP'(d_1,\ldots,d_n)$. Therefore $\CP(d_1,\ldots,d_n)\subseteq \CP'(d_1,\ldots,d_n)$.

On the other hand, $\CP'(d_1,\ldots,d_n)$ is a zonotope with edges parallel to the edges of the standard permutahedron $\CP(1,\ldots,1)$. It follows e.g. from \cite[Section 9]{BR} that the vertices of $\CP'(d_1,\ldots,d_n)$ are in bijection with the vertices of  $\CP(1,\ldots,1)$, and are given by $p_{\sigma}$. So $\CP'(d_1,\ldots,d_n)\subset \CP(d_1,\ldots,d_n)$.
\end{proof}

\begin{example}
For $n=3$ we get three segments $[(d_1,0,0),(0,d_1,0)]$, $[(d_1,0,0),(0,0,d_1)]$ and $[(0,d_2,0),(0,0,d_2)]$.
\end{example}

\begin{rem}
Quite surprisingly, a similar polytope appeared in a recent work of Alishahi, Liu and the first author on Heegaard Floer homology \cite{AGL}.
\end{rem}

\subsection{Generalized \texorpdfstring{$(q,t)$}{(q,t)}-Catalans for \texorpdfstring{$n=3$}{n=3}.}

Given $d_1\le d_2$, we can consider the Young diagram $\lambda_{d_1,d_2}=(d_1+d_2,d_1)$. We will draw Young diagrams in French notation, with the corner at $(0,0)$, see Figure \ref{fig: triangular}. We also consider the line 
$$
\ell_{d_1,d_2}=\{x+d_2y=d_1+2d_2+\varepsilon\}
$$
where $\varepsilon$ is a small positive number. The following lemma  shows that $\lambda_{d_1,d_2}$ is a triangular partition in the sense of \cite{BM}.

\begin{lemma}
The diagram $\lambda_{d_1,d_2}$ is the largest Young diagram below the line $\ell_{d_1,d_2}$. If a diagram $\mu$ is strictly below $\ell_{d_1,d_2}$, then $\mu\subset \lambda_{d_1,d_2}$.
\end{lemma}

\begin{proof}
Let us describe all integer points $(x,y)$ satisfying $x+d_2y< d_1+2d_2+\varepsilon$. For $y=1$ we get $x<d_1+d_2+\varepsilon$, so $x\le d_1+d_2$. For $y=2$ we get $x<d_1+\varepsilon$, so $x\le d_1$. For $y\ge 3$ we get $x<d_1+(2-y)d_2+\varepsilon<0$, so there are no integer points (here we used $d_1\le d_2$). The result follows.
\end{proof}

\begin{figure}[ht!]
\begin{tikzpicture}
\draw [->] (0,0)--(0,4);
\draw [->] (0,0)--(14,0);
\draw (0,0)--(8,0);
\draw (0,1)--(8,1);
\draw (0,2)--(3,2);
\draw (0,0)--(0,2);
\draw (1,0)--(1,2);
\draw (2,0)--(2,2);
\draw (3,0)--(3,2);
\draw (4,0)--(4,1);
\draw (5,0)--(5,1);
\draw (6,0)--(6,1);
\draw (7,0)--(7,1);
\draw (8,0)--(8,1);
\draw (0,2.64)--(13.2,0);
\end{tikzpicture}
\caption{The line $\ell_{d_1,d_2}$ and the diagram $\lambda_{d_1,d_2}$ for $(d_1,d_2)=(3,5)$.}
\label{fig: triangular}
\end{figure}
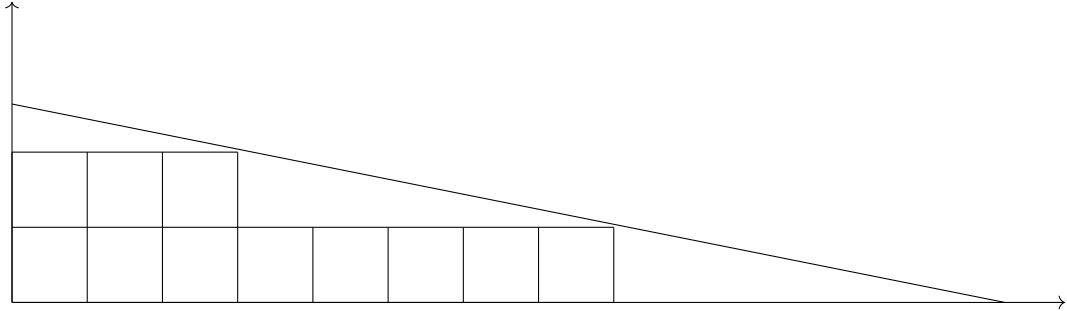

Following \cite{BM}, we define two statistics on subdiagrams of $\lambda_{d_1,d_2}$.

\begin{defn}
Given $\mu\subset \lambda_{d_1,d_2}$, we define $\area(\mu)=|\lambda_{d_1,d_2}|-|\mu|$
and
$$
\dinv(\mu)=\left\{\sq\in \mu:\frac{a(\sq)}{\ell(\sq)+1}\le d_2< \frac{a(\sq)+1}{\ell(\sq)}\right\}
$$
Here $a(\sq)$ and $\ell(\sq)$ are respectively the arm and the leg of a box $\sq$ in $\mu$.
\end{defn}
 
Note that $d_2$ in the definition of $\dinv$ is negative reciprocal to the slope of the line $\ell_{d_1,d_2}$.

\begin{thm}
\label{thm: phi}
The map $\phi:\mu\mapsto (\area(\mu),\dinv(\mu))$ yields a bijection between the subdiagrams $\mu\subset \lambda_{d_1,d_2}$ and the integer points in the trapezoid 
\begin{equation}
\label{eq: trapezoid}
\{d_1+d_2\le x+y\le 2d_1+d_2,\ 2x+y\le 2d_1+d_2,\ x+2y\le 2d_1+d_2\}.
\end{equation}
As a consequence, we get the generalized $(q,t)$-Catalan number
\begin{equation}
\label{eq: qtcat}
\sum_{\mu\subset \lambda_{d_1,d_2}}q^{\area(\mu)}t^{\dinv(\mu)}=F(d_1,d_2).
\end{equation}
\end{thm}

Equation \eqref{eq: qtcat} is a special case of the main result of \cite{Blasiak}, but we give a more direct proof here generalizing \cite[Theorem 4.1]{GM2}.

\begin{proof}
Let us write $\mu=(a+b,a)$, then $0\le a\le d_1$ and $0\le b\le d_1+d_2-a$. We have the following cases (see Figure \ref{fig: phi}):

(a) If $a+b\le d_2$ then $\dinv(\mu)=a+b$, so $$\phi(\mu)=(2d_1+d_2-2a-b,a+b)=(x,y).$$ We have $a=2d_1+d_2-x-y, b=2y+x-2d_1-d_2$, so the inequalities $0\le a,a\le d_1,0\le b$ and $a+b\le d_2$ respectively translate to the inequalities $x+y\le 2d_1+d_2,x+y\ge d_1+d_2, 2y+x\ge 2d_1+d_2$ and $y\le d_2$. 

(b) If $a+b>d_2, b\le d_2$ then $\dinv(\mu)=2a+2b-d_2$, so $$\phi(\mu)=(2d_1+d_2-2a-b,2a+2b-d_2)=(x,y).$$ We have $y=d_2\mod 2$, and 
$$
a=\frac{4d_1+d_2-2x-y}{2},\ b=x+y-2d_1
$$
The inequalities $0\le a,a\le d_1,0\le b, b\le d_2$ and $a+b>d_2$ respectively translate to the inequalities
$$
2x+y\le 4d_1+d_2,\ 2x+y\ge 2d_1+d_2,\ x+y\ge 2d_1,\ x+y\le 2d_1+d_2,\ y>d_2.
$$
The second, fourth and fifth inequalities define a triangle $\mathbf{T}$ with vertices $(0,2d_1+d_2),(d_1,d_2)$ and $(2d_1,d_2)$ with the bottom side removed. The other two inequalities are satisfied on this triangle. In other words, in this case the image of $\phi$ is the set  of all integer points in the triangle $\mathbf{T}$ satisfying $y=d_2\mod 2$.

(c) If $b>d_2$ then $\dinv(\mu)=2a+d_2+1$, so $$\phi(\mu)=(2d_1+d_2-2a-b,2a+d_2+1)=(x,y).$$ We have $y=d_2+1\mod 2$, and
$$
a=\frac{y-d_2-1}{2},\ b=2d_1+2d_2+1-x-y.
$$
The inequalities $0\le a,a\le d_1,b>d_2$ and $a+b\le d_1+d_2$ respectively translate to the inequalities 
$$
y\ge d_2+1,\ y\le 2d_1+d_2+1,\ x+y<2d_1+d_2+1,\ 2x+y\ge 2d_1+d_2+1.
$$
Similarly to (b), the image of $\phi$ in this case is the set of all integer points in the triangle $\mathbf{T}$ 
satisfying $y=d_2+1\mod 2$.
\end{proof}

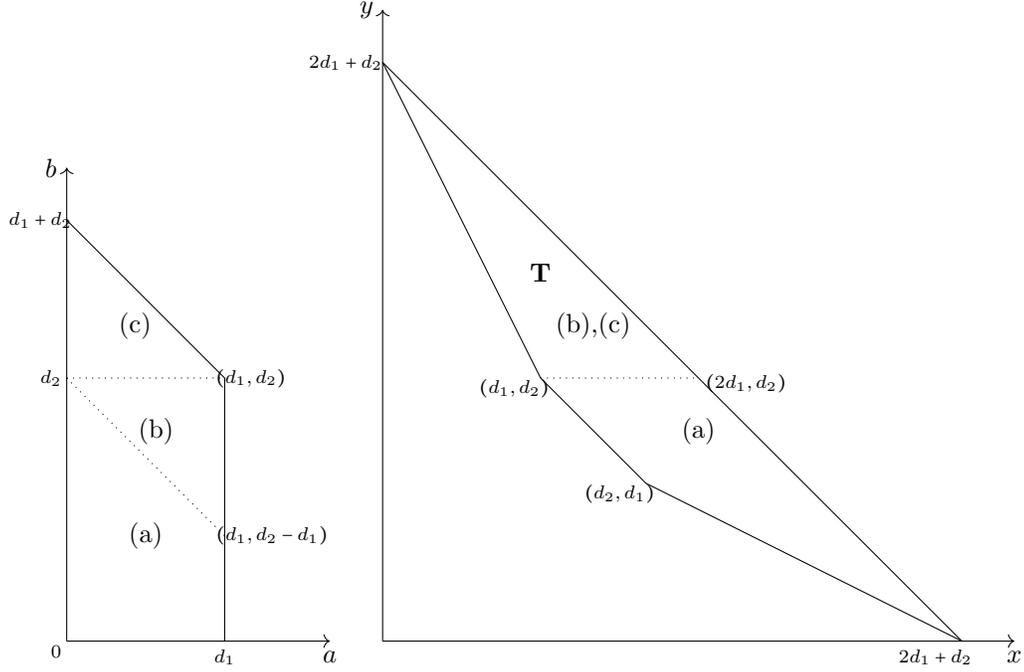
\begin{figure}
\begin{tikzpicture}[scale=0.7]
\draw[->] (0,0)--(0,9);
\draw[->] (0,0)--(5,0);
\draw (3,0)--(3,5)--(0,8);
\draw (5,-0.3) node {$a$};
\draw (-0.3,9) node {$b$};
\draw (-0.5,8) node {\scriptsize $d_1+d_2$};
\draw (3,-0.3) node {\scriptsize $d_1$};
\draw (3.5,5) node {\scriptsize $(d_1,d_2)$};
\draw (-0.3,5) node {\scriptsize $d_2$};
\draw (3.9,2) node {\scriptsize $(d_1,d_2-d_1)$};
\draw (-0.2,-0.2) node {\scriptsize $0$};
\draw [dotted] (0,5)--(3,5);
\draw [dotted] (0,5)--(3,2);
\draw (1.5,2) node {(a)};
\draw (1.7,4) node {(b)};
\draw (1.3,6) node {(c)};

\draw [->] (6,0)--(18,0);
\draw [->] (6,0)--(6,12);
\draw (6,11)--(17,0);
\draw (6,11)--(9,5)--(11,3)--(17,0);
\draw [dotted] (9,5)--(12,5);
\draw (5.7,12) node {$y$};
\draw (18,-0.3) node {$x$};
\draw (5.3,11) node {\scriptsize $2d_1+d_2$};
\draw (16.5,-0.3) node {\scriptsize $2d_1+d_2$};
\draw (8.5,4.8) node {\scriptsize $(d_1,d_2)$};
\draw (10.5,2.8) node {\scriptsize $(d_2,d_1)$};
\draw (12.9,4.9) node {\scriptsize $(2d_1,d_2)$};
\draw (12,4) node {(a)};
\draw (9,7) node {$\mathbf{T}$};
\draw (10,6) node {(b),(c)};
\end{tikzpicture}
\caption{The bijection $\phi$}
\label{fig: phi}
\end{figure}

Finally, we compare the above combinatorial results with the fundamental domain. Observe that for $n=3$ the fundamental domain $\CP(d_1,d_2)$ is a hexagon with an axis of symmetry, which cuts it into two equal halves (see Figure \ref{fig: domain}).

\begin{thm}
\label{thm: fund domain to gens}
The integer points in a half of $\CP(d_1,d_2)$ (including boundary) are in bijection with the generators of the ideal $\cJ(d_1,d_2)$.
\end{thm}

\begin{proof}
We construct the desired bijection in several steps:

1) The integer points in a half of $\CP(d_1,d_2)$ are in bijection with the subdiagrams of $\lambda_{d_1,d_2}$. 
Indeed, we can write such points as $(d_1,0,d_1+d_2)+a(1,0,-1)+b(0,1,-1)$. It is easy to see by comparing Figures \ref{fig: domain} and \ref{fig: phi} that $0\le a\le d_1$ and $0\le b\le d_1+d_2-a$, and hence $(a,b)$ define a subdiagram $\mu=(a+b,a)$.

2) By Theorem \ref{thm: phi}, we have a bijection $\phi$ between the subdiagrams of $\lambda_{d_1,d_2}$ and the points in the trapezoid \eqref{eq: trapezoid}.

3) By Proposition \ref{prop: bijection}, there is a bijection between the generators of $\cJ(d_1,d_2)$ and the points in the trapezoid \eqref{eq: trapezoid}.
\end{proof}

We expect that the bijection in Theorem \ref{thm: fund domain to gens} is far more than a combinatorial coincidence. In particular, by tracing through the bijections we see that $\dinv$ defines a piecewise linear function on the fundamental domain $\CP(d_1,d_2)$, and we expect this function to be closely related to the dimension of cells in an appropriately chosen cell decomposition of $\Sp_{\gamma}$. The corresponding (equivariant) homology classes of the cells would then correspond to some elements of $\cJ(d_1,d_2)$, and we expect that these would indeed generate the ideal. We plan to study these questions and generalize them to $n>3$ in future work. 

\end{appendix}

\end{document}